\documentclass[11pt,twoside]{article}
\usepackage{amsmath,amssymb,epsfig,subfigure,graphicx,epstopdf,color}
\usepackage{multirow}
\usepackage{algorithm}
\usepackage{algorithmic}
\usepackage[titletoc,title]{appendix}
\newtheorem{proposition}{Proposition}[section]
\newtheorem{theorem}[proposition]{Theorem}
\newtheorem{lemma}[proposition]{Lemma}

\newtheorem{remark}[proposition]{Remark}
\newtheorem{example}[proposition]{Example}

\newtheorem{assumption}[proposition]{Assumption}

\newtheorem{procedure}[proposition]{Procedure}
\newenvironment{proof}{{\noindent \em Proof.}}{\hfill $\fbox{}$ \vspace*{5mm}}
\numberwithin{equation}{section}

\newcommand{\ba}{{\bf a}}

\newcommand{\bz}{{\bf z}}

\newcommand{\la}{\lambda}
\newcommand{\La}{\Lambda}

\newcommand{\ca}{\mathcal{A}}

\newcommand{\ck}{\mathcal{K}}

\newcommand{\cm}{\mathcal{M}}

\newcommand{\co}{\mathcal{O}}

\newcommand{\ce}{\mathcal{E}}

\newcommand{\cx}{\mathcal{X}}

\newcommand{\cy}{\mathcal{Y}}

\newcommand{\rd}{{\mathrm {D}}}
\newcommand{\diag}{{\rm diag}}

\newcommand{\dist}{{\rm dist}}

\newcommand{\grad}{{\rm grad\;}}

\newcommand{\ve}{{\rm vec}}
\newcommand{\veh}{{\rm vech}}

\newcommand{\tr}{{\rm tr}}
\newcommand{\bd}{\boldsymbol}

\newcommand{\R}{{\mathbb R}}

\newcommand{\Rn}{{\mathbb R}^n}
\newcommand{\Rnn}{{\mathbb R}^{n\times n}}

\newcommand{\SRn}{\mathbb{SR}^{n\times n}}
\newcommand{\BE}{\begin{equation}}
\newcommand{\EE}{\end{equation}}

\DeclareMathOperator*{\argmin}{argmin}
\newcommand{\normmm}[1]{{\vert\kern-0.25ex \vert\kern-0.25ex \vert #1
    \vert\kern-0.25ex \vert\kern-0.25ex\vert}}

\begin{document}

\title{A Riemannian  Inexact Newton Dogleg Method for Constructing  a Symmetric Nonnegative Matrix with Prescribed Spectrum}
\author{Zhi Zhao\thanks{Department of Mathematics, School of Sciences, Hangzhou Dianzi University, Hangzhou 310018,
People's Republic of China (zhaozhi231@163.com). The research of this author was supported by the National Natural Science
Foundation of China (No. 11601112) and the Zhejiang Provincial Natural Science Foundation of China (No. LY21A010010).}
\and Teng-Teng Yao\thanks{Department of Mathematics, School of Sciences, Zhejiang University of Science and Technology, Hangzhou 310023, People's Republic of China (yaotengteng718@163.com). The research of this author was supported by
the National Natural Science Foundation of China (No. 11701514) and the Zhejiang Provincial Natural Science Foundation
of China (No. LY21A010004).}
\and Zheng-Jian Bai\thanks{Corresponding author. School of Mathematical Sciences and Fujian Provincial Key Laboratory on Mathematical Modeling \& High Performance Scientific Computing,  Xiamen University, Xiamen 361005, People's Republic of China (zjbai@xmu.edu.cn). The research of this author was partially supported by the National Natural Science Foundation of China (No. 11671337) and the Natural Science Foundation of Fujian Province of China (No. 2021J01033).}
\and Xiao-Qing Jin\thanks{Department of Mathematics, University of Macau, Macao, People's Republic of China (xqjin@umac.edu.mo).
The research of this author was supported by the research grants MYRG2019-00042-FST and CPG2021-00035-FST
from University of Macau and 0014/2019/A from FDCT.}
 }
\maketitle

\begin{abstract}
This paper is concerned with the inverse problem of constructing a symmetric nonnegative matrix from realizable spectrum.
We reformulate the inverse problem as an underdetermined nonlinear matrix equation
over a Riemannian product manifold.  To solve it, we develop a Riemannian underdetermined inexact Newton dogleg method for solving a general underdetermined nonlinear equation defined between Riemannian manifolds and Euclidean spaces. The global and quadratic convergence of the proposed  method is established under some mild assumptions. Then we solve the inverse problem by applying the proposed method to its equivalent nonlinear matrix equation and a preconditioner for the perturbed normal Riemannian Newton equation is also constructed. Numerical tests show the efficiency of the proposed  method for solving the inverse problem.
\end{abstract}

{\bf Keywords.}  Symmetric nonnegative inverse eigenvalue problem, underdetermined equation,
Riemannian Newton dogleg method, preconditioner.

\vspace{3mm}
{\bf AMS subject classifications.} 15A18, 65F08, 65F18, 65F15.

\section{Introduction}

An $n$-by-$n$ matrix $A$ is nonnegative if all its entries are all nonnegative, i.e., $(A)_{ij}\ge 0$ for all $i,j=1,\ldots,n$, where $(A)_{ij}$ means the $(i,j)$th entry of $A$. Nonnegative matrices arise in a wide variety  of applications such as finite Markov chains, probabilistic algorithms, graph theory, the linear complementarity problems, matrix scaling, and input-output analysis in economics,    etc (see for instance \cite{BR97,BP79,M88,SA06}). The nonnegative inverse eigenvalue problem (NIEP) is a structured inverse eigenvalue problem \cite{CG02,CG05,X98}, which aims to determine whether a given self-conjugate set of complex numbers is the spectrum of a nonnegative matrix. Various theoretical results have been obtained on the existence theory of the NIEP
in the literature \cite{ELN04,JMPP17,JP16,K51,LS06,LL78,O83,R96,S06-2,S13}.

This paper is concerned with the symmetric NIEP of constructing  a symmetric nonnegative matrix from a realizable spectrum
numerically. Recall that a list of complex numbers which occurs as the spectrum of some nonnegative matrix
is called a realizable spectrum \cite{JMPP17}. The inverse eigenvalue problem of reconstruction of
a real symmetric nonnegative matrix from a prescribed realizable spectrum can be stated as follows:

{\bf SNIEP.} {\em Given a realizable list of $n$ real numbers $\{\lambda_1, \lambda_2, \ldots , \lambda_n\}$,
find an $n$-by-$n$ real symmetric nonnegative matrix $A$ such that its eigenvalues are
$\lambda_1, \lambda_2, \ldots , \lambda_n$.}

There exist some  numerical methods for solving the NIEP including constructive methods \cite{JP16,Papa16,S83},
recursive methods \cite{ES16,L13}, isospectral gradient flow approaches \cite{CL11,CDS04,CD91,CG98},
alternating projection algorithm \cite{O06} and Riemannian inexact Newton method \cite{ZBJ18}.
Constructive methods and recursive methods have special requirements on the realizable spectrum, and thus these methods are restricted to solving the NIEPs with additional constraints on the realizable spectrum. Isospectral gradient approaches and alternating projection algorithms can be used in the solution of medium-scale problems. The Riemannian inexact Newton method
can be applied to solve large-scale problems, which depends heavily on how to solve the Riemannian Newton equation efficiently. This motivates us to find an effective preconditioner to improve the efficiency of the proposed Riemannian Newton method for solving large-scale SNIEPs.

In the past few decades, various numerical methods have been proposed for finding zeros of
underdetermined nonlinear maps defined between Euclidean spaces (see for instance \cite{BLSYG17,CY94,ESV11,ESV12,FKM05,Martin91,S06,WW90}).
However, to our knowledge, except for the Riemannian inexact Newton method proposed in \cite{ZBJ18}, there exist few other effective numerical algorithms in the literature for finding the zeros of  general underdetermined maps between a Riemannian manifold and a Euclidean space.

In this paper,  based on the symmetric Schur decomposition,  we reformulate the SNIEP as a problem of finding  a solution of an underdetermined nonlinear matrix equation over a product Riemannian manifold. To solve it, we first develop a Riemannian inexact Newton dogleg method for solving a general underdetermined nonlinear equation over a Riemannian manifold. This is motivated by the three papers due to  Pawlowski et al. \cite{PS08}, Simons \cite{S06},  and Zhao et al. \cite{ZBJ18}.
In \cite{S06}, Simons provided an exact trust region method (i.e., underdetermined Newton dogleg
method) for finding zeros of underdetermined nonlinear maps defined between Euclidean spaces.
In \cite{PS08}, Pawlowski et al.  presented inexact Newton dogleg methods for solving nonlinear equations defined on a Euclidean space. In \cite{ZBJ18}, Zhao et al. gave a Riemannian inexact Newton method for constructing a nonnegative matrix with prescribed realizable spectrum. The global and quadratic convergence of the proposed method is established under some mild assumptions.  Then we find a solution to the SNIEP by applying the proposed method to its corresponding underdetermined nonlinear matrix equation over a product Riemannian manifold.
To further improve the efficiency, by exploring the structure property of the SNIEP, a preconditioning technique is presented,
which can also be combined with the Riemannian inexact Newton method in \cite{ZBJ18} for solving the SNIEP.
Finally, we report some numerical experiments to demonstrate that the proposed method with
the constructed preconditioner can solve the  SNIEP efficiently.

Throughout this paper, we use the following notation. The symbols $A^T$ and $A^H$ denote the transpose and conjugate transpose of a matrix $A$, respectively. $I_n$ denotes  the identity matrix of order $n$. Let $\Rnn$ and $\SRn$ be the set of all $n$-by-$n$  real matrices and the set of all $n$-by-$n$ real symmetric matrices, respectively. Let $\Rnn_{+}$ and $\SRn_+$ denote the nonnegative orthants of $\Rnn$ and $\SRn$, respectively. $\|\cdot \|_F$ stands for  the matrix Frobenius norm.  Denote by $A \odot B$ and $[A,B]:=AB-BA$  the Hadamard product and Lie Bracket of two $n$-by-$n$ matrices $A$ and $B$, respectively.  Denote by $\tr(A)$ the sum of the diagonal entries of a square matrix $A$. $\diag(\ba)$ is a diagonal matrix whose $i$th diagonal element is the $i$th component of a vector $\ba$. For a matrix $A\in \Rnn$, let $\ve(A)$ be the vectorization of $A$, i.e., a column vector obtained by stacking the columns of $A$ on top of one another, and define $\veh(A) \in \mathbb{R}^{n(n+1)/2}$ by
\[
\big(\veh(A)\big)_{{\frac{(j-1)j}{2}+i}} := (A)_{ij}, \quad 1\le i\le j\le n.
\]
Let $\cx$ and $\cy$ be two finite-dimensional vector spaces  equipped with a scalar inner product $\langle\cdot,\cdot\rangle$ and its induced norm $\|\cdot\|$. Let $\ca:\cx\to \cy$ be a linear operator such that $\ca [x]\in\cy$ for all $x\in\cx$,  and the adjoint of $\ca$ is denoted by $\ca^*$.  Define the operator norm of $\ca$ by $\normmm{\ca}:=\sup\{ \| \ca [x]\| \ | \ x\in\cx\mbox{ with } \|x\|=1\}.$

The rest of this paper is organized as follows. In Section \ref{sec2} the SNIEP is written as an underdetermined nonlinear matrix equation over a Riemannian product manifold.  In Section \ref{sec3} we develop a Riemannian inexact Newton dogleg method for solving a general underdetermined nonlinear equation over a Riemannian manifold. The global and  quadratic convergence of the proposed method is established under some mild assumptions.
In Section \ref{sec4} we apply the Riemannian inexact Newton dogleg method developed in Section \ref{sec3} to the  SNIEP, where an effective preconditioner is also provided. Finally, some numerical experiments and concluding remarks
are given in Sections \ref{sec5} and \ref{sec6}, respectively.
\section{Reformulation}  \label{sec2}
In this section, we reformulate the SNIEP as an equivalent problem of solving a specific underdetermined nonlinear matrix equation over a Riemannian product manifold. Let $\Lambda$ be the diagonal matrix defined by
\[
\Lambda := \diag (\bd\la)\in \Rnn,\quad \bd\la:=(\lambda_{1}, \lambda_{2}, \ldots, \lambda_n)^T \in \Rn.
\]
Define the orthogonal group $\co(n)$ by
\[
\co(n): =  \big \{ Q \in \mathbb{R}^{n\times n} \ | \ Q^TQ = I_{n} \big \}.
\]
The set $\SRn_+$ can be represented by
\[
\SRn_+=\{S \odot S\in\Rnn \; | \; S\in\SRn\}.
\]
Based on the symmetric Schur decomposition \cite{GV13}, the smooth manifold of isospectral matrices for $\SRn$ is given by
\[
\cm(\La):=\{A=Q\Lambda Q^T\in\SRn \; | \; Q\in\co(n)\}.
\]
Hence, the SNIEP has a solution if and only if $\cm(\La)\cap \SRn_+\neq\emptyset$.

Suppose the SNIEP has at least one solution. Then the SNIEP is reduced to the following constrained matrix equation:
\BE\label{RNEWEQ1}
\begin{array}{c}
\Phi(S,Q) := S \odot S - Q\Lambda Q^T = \mathbf{0}_{n\times n} ,\quad   \mbox {s.t. $(S,Q) \in \SRn \times \co(n)$},
\end{array}
\EE
where $\mathbf{0}_{n\times n}$ means the zero matrix of order $n$.

We note that if $(\overline{S}, \overline{Q})\in\SRn \times \co(n)$ is a solution to  (\ref{RNEWEQ1}), then $\overline{C}:=\overline{S}\odot \overline{S}$ is a solution to the SNIEP. To avoid confusion, we refer to (\ref{RNEWEQ1}) as the SNIEP.

We point out that $\Phi:\SRn \times \co(n)\to\SRn$ is  a smooth mapping from the product manifold $\SRn \times \co(n)$ to the Euclidean space $\SRn$. It is obvious that the dimension of $\SRn \times \co(n)$ is larger than  the dimension of $\SRn$ for  $n\geq 2$. This shows that the matrix equation $\Phi(S,Q) = \mathbf{0}_{n\times n}$ defined by (\ref{RNEWEQ1}) is underdetermined for $n\geq 2$.
\section{General underdetermined nonlinear equation over Riemannian manifold}\label{sec3}

In this section, we consider a general  underdetermined nonlinear equation, where the nonlinear map is a differentiable mapping  between a Riemannian manifold and a Euclidean space. Then we introduce  a Riemannian  inexact Newton dogleg method for solving the underdetermined nonlinear equation.
The global and  quadratic convergence is also established under some mild assumptions.
\subsection{Problem statement}
Let $\cm$  and $\ce$ be respectively a Riemannian manifold and  a Euclidean space with ${\rm dim}(\cm) > {\rm dim}(\ce)$.
Let $F:\cm\to\ce$ be a differentiable nonlinear mapping between $\cm$ and $\ce$. In this subsection, we focus on the following  underdetermined nonlinear equation:
\BE\label{eq:nleq}
F(x) = 0, \qquad  \text{subject to (s.t.)}  \;x \in \cm,
\EE
where $0$ is the zero vector of $\ce$.

For simplicity, let $\langle \cdot, \cdot \rangle$ denote the Riemannian metric on $\cm$ and
the inner product on $\ce$ with its induced norm $\|\cdot\|$. Denote by  $T_x\cm$  the tangent space of $\cm$ at a point $x\in\cm$. Let $\rd F(x):T_x\cm\to T_{F(x)}\ce\simeq\ce$ be the differential (derivative) of $F$ at $x\in\cm$ \cite[p.38]{AMS08}, where ``$\simeq$" means the identification of two sets. Then a point $x\in \cm$ is called a {\em stationary point} of $F$ if
\[
\|F(x)\| \leq \|F(x) + \rd F(x)[\Delta x] \|,\quad \forall \Delta x\in T_x\cm.
\]
Define the  merit function $f:\cm\to\R$ by
\BE\label{fF}
f(x):=\frac{1}{2}\|F(x)\|^{2},\quad \forall x\in \cm.
\EE
By hypothesis, $F$ is differentiable. Then the function $f:\cm \to \R$ is also differentiable. As in \cite[p. 46]{AMS08}, the Riemannian gradient $\grad f(x)$ of $f$ at $x\in\cm$ is defined as the unique element in $T_x\cm$  such that
\[
\langle \grad f(x), \xi_x\rangle=\rd f(x)[\xi_x],\quad \forall \xi_x\in T_x\cm.
\]
If follows from \eqref{fF} that the Riemannian gradient of $f$ at $x\in \cm$ is give by \cite[p.185]{AMS08}:
\BE\label{GRAVEC}
\grad f(x) = (\rd F(x))^*[F(x)],
\EE
where $(\rd F(x))^*:T_{F(x)}\ce\to T_{x}\cm$ is the adjoint operator of $\rd F(x)$.
Specially, $x\in\cm$ is a stationary point of $F$ if and only if $x\in \cm$ is a stationary point of $f$, i.e.,
$\grad f(x) = (\rd F(x))^*[F(x)] = 0_x$, where $0_x$ is the zero tangent vector of $T_x\cm$.
\subsection{Riemannian inexact Newton dogleg method}
In the following, we develop a Riemannian trust region method for solving (\ref{eq:nleq}). Let $R$ be a retraction on $\cm$ \cite[p.55]{AMS08}. As in \cite{PS08,S06}, given the current point $x_k\in \cm$, we consider the following linear model of
the nonlinear map $F$ at $x_k\in \cm$:
\BE\label{llm}
F(x_k) + \rd  F(x_k) [ \xi_{k} ], \qquad \;  \xi_{k} \in T_{x_k}\cm.
\EE
Let $\Delta x_k\in T_{x_k}\cm$ be an exact or approximate minimizer of the following  trust region least square problem:
\BE\label{ls:dx}
\min_{ \xi_{k} \in T_{x_k}\cm,\; \|\xi_{k}\|\le\delta_k}\|F(x_k) + \rd   F(x_k) [ \xi_{k} ]\|,
\EE
where $\delta_k>0$ is the trust region radius.
The actual reduction and predicted reduction induced by $\Delta x_k$ at the current point $x_k\in \cm$ are defined by
\BE\label{def:ared}
{\rm Ared}_{k}(\Delta x_k):=\| F(x_k)\|-\| F(R_{x_k}(\Delta x_k))\|
\EE
and
\BE\label{def:pred}
{\rm Pred}_{k}(\Delta x_k):=\| F(x_k)\|-\| F(x_k)+\rd F(x_k)[\Delta x_k]\|.
\EE
Then the Ared/Pred condition needs to be tested, i.e., whether $\Delta x_k$ satisfies
the following condition
\BE\label{Ared/Predcondition}
\frac{{\rm Ared}_{k}(\Delta x_k)}{{\rm Pred}_{k}(\Delta x_k)}
= \frac{\| F(x_k)\|-\| F(R_{x_k}(\Delta x_k))\| }{ \| F(x_k)\|-\| F(x_k)+\rd F(x_k)[\Delta x_k]\| } \geq t,
\EE
where $0<t<1$ is given constant. If the Ared/Pred condition is satisfied by $\Delta x_k$, then define $x_{k+1}:= R_{x_k}(\Delta x_k)$
and compute $\delta_{k+1}$ by a prescribed rule.
If not, the trust region radius $\delta_k$ is shrunk and we need find a new tangent vector $\Delta x_k\in T_{x_k}\cm$ within the trust region.

We note that the nonlinear equation $F(x) = 0$ is underdetermined. Hence, the global minimizer of \eqref{ls:dx} is not unique.
To calculate a suitable $\Delta x_k\in T_{x_k}\cm$, we can generalize the idea of exact trust region method for solving
underdetermined equation between Euclidean spaces \cite[p.34]{S06} in the following way
\BE\label{EXACTTR}
\Delta x_k: = \argmin_{ \xi_k \bot {\rm null}(\rd F(x_k)), \; \| \xi_k \| \leq \delta_k} \| F(x_k) +\rd  F(x_k) [\xi_k ] \|,
\EE
where ${\rm null}(\cdot)$ means the null space of a linear mapping.
In general, the computation of $\Delta x_k$ by (\ref{EXACTTR}) is costly for large-scale problems.

In this paper, we generalize the underdetermined  dogleg method in \cite[p.42]{S06}, which was presented for solving an underdetermined nonlinear equation defined between Euclidean spaces, to the solution of \eqref{eq:nleq} over $\cm$. Suppose that $\grad f(x_k) \neq 0_{x_k}$,
the Cauchy point at $x_k\in \cm$ is defined to be the minimizer of
$\frac{1}{2}\|F(x_k)+\rd F(x_k)[\Delta x_k]\|^{2}$
along the steepest descent direction $-\grad f(x_k)=-(\rd F(x))^*[F(x)]$,
which is denoted by $\Delta x_k^{CP}$, i.e.,
\BE\label{CP2}
\Delta x_k^{CP}:=-\frac{\|(\rd F(x_k))^{*}[F(x_k)]\|^2}
{\|\rd F(x_k)\circ(\rd F(x_k))^{*}[F(x_k)]\|^2}(\rd F(x_k))^{*}[F(x_k)]\in T_{x_k}\cm.
\EE
Specially, we have
\BE\label{CPOT}
\Delta x_k^{CP} \perp {\rm null} (\rd F(x_k)).
\EE
The Riemannian Newton point $\Delta x_k^{N}$ is defined by
\BE\label{INNEWTON}
\begin{array}{cc}
\Delta x_k^{N}: = \argmin\limits_{ \xi_k \perp {\rm null} (\rd F(x_k))} \| F(x_k) +\rd  F(x_k) [\xi_k ] \|\in T_{x_k}\cm.
\end{array}
\EE
The dogleg curve $\Gamma^{DL}_k$ is defined to be the piecewise linear curve
joining the origin $0_{x_k}$, the Cauchy point $\Delta x_k^{CP}$, and
the Riemannian Newton point $\Delta x_k^{N}$.
Similar to the analysis in \cite[pp. 42-44]{S06}, the norm of the linear model $F(x_k)+ \rd F(x_k) [ \xi_{x_k} ]$
is monotone decreasing along the dogleg $\Gamma^{DL}_k$.
By (\ref{CPOT}) and (\ref{INNEWTON}) we have
\BE\label{XKOT}
\Delta x_k  \perp {\rm null} (\rd F(x_k)), \quad \forall \Delta x_k\in \Gamma^{DL}_k.
\EE
The Riemannian dogleg step aims to find the tangent vector $\Delta x_k$ such that
\[
\Delta x_k: = \argmin_{ \xi_k \in \Gamma^{DL}_k, \; \| \xi_k \| \leq \delta_k} \| F(x_k) +\rd  F(x_k) [\xi_k ] \|.
\]
The above minimization problem has a unique minimizer, which can be calculated explicitly.
The dogleg method is a special inexact trust region method, which is often computationally
efficient than the exact trust region method.
However, the Newton point $\Delta x_k^{N}$ is still computationally costly for large-scale problems.
Based on (\ref{EXACTTR}), (\ref{XKOT}), and the analysis in \cite{S06}, the orthogonality of $\Delta x_k$
with the null space of $\rd F(x_k)$ is essential for the convergence analysis.

In \cite{PS08}, inexact Newton dogleg methods were given for solving
nonlinear equations defined on Euclidean spaces.
To generalize these methods directly  to the solution of (\ref{eq:nleq}),
we need to find an inexact Newton point $\Delta x_k^{IN}\in T_{x_k}\cm$  such that
\BE\label{INNID}
\frac{\|F(x_k)+\rd F(x_k)[\Delta x_k^{IN}]\|}{\|F(x_k)\|} < \eta_k < \eta_{\max} < 1
\quad \mbox{and} \quad \Delta x_k^{IN} \perp {\rm null}(\rd F(x_k)),
\EE
where $\eta_k$ is a forcing term \cite{EW96}.
However, if the differential $\rd F(x_k):T_{x_k}\cm \to T_{F(x_k)}\ce$ is not surjective,
the first condition in  (\ref{INNID}) may not be attainable.
The Riemannian Newton point $T_{x_k}\cm \ni\Delta x_k^N \perp {\rm null} (\rd F(x_k))$ defined by (\ref{INNEWTON}) is
the minimum norm solution of the least squares problem
\[
\min\limits_{\Delta x_k\in T_{x_k}\cm} \|F(x_k) + \rd  F(x_k) [\Delta x_k^N ]\|,
\]
which is given in the form of
\[
\Delta x_k^N = -(\rd F(x_k))^\dag F(x_k),
\]
where $(\rd F(x_k))^\dag$ denotes the pseudoinverse of the linear operator $\rd F(x_k)$ \cite[pp. 163--164]{L69}. We note that
\[
(\rd F(x_k))^\dag=\lim_{\sigma\to 0^+}(\rd F(x_k))^*\circ \big( \rd F(x_k) \circ (\rd F(x_k))^*
+ \sigma \mathrm{id}_{T_{F(x_k)}\mathcal{E}}\big)^{-1},
\]
where $\mathrm{id}_{T_{F(x_k)}\mathcal{E}}$ is the identity operator
on $T_{F(x_k)}\mathcal{E}$.  This motivates us to solve the following perturbed Riemannian normal equation
\BE\label{RNEWTONEQ}
\Big(\rd F(x_k)\circ(\rd F(x_k))^* +\sigma_k \mathrm{id}_{T_{F(x_k)}\mathcal{E}}\Big)
[ \Delta z_k ] = - F(x_k),
\EE
for $\Delta z_k\in T_{F(x_k)}\ce$, where $\sigma_k >0$ is a given constant.
We observe that
$$\rd F(x_k)\circ(\rd F(x_k))^* +\sigma_k \mathrm{id}_{T_{F(x_k)}\mathcal{E}}$$
is a self-adjoint positive definite linear operator defined on the Euclidean space $\ce$. Therefore,
we can  solve  (\ref{RNEWTONEQ}) inexactly by using the conjugate gradient (CG) method \cite{GV13}. Moreover, once an approximate solution $\Delta z_k$ is obtained,
the inexact Newton point is given by $\Delta x_k^{IN}:= (\rd F(x_k))^*[ \Delta z_k ]$, which satisfies the second condition in (\ref{INNID}) naturally.

Therefore, the inexact dogleg curve $\widehat{\Gamma}^{DL}_k$ is defined to be the piecewise linear curve
joining the origin $0_{x_k}$, the Cauchy point $\widehat{\Delta x}_k^{CP}$ defined by
\BE\label{CP}
\widehat{\Delta x}_k^{CP}:=-\frac{\|(\rd F(x_k))^{*}[F(x_k)]\|^{2}}
{\|\rd F(x_k)\circ(\rd F(x_k))^{*}[F(x_k)]\|^{2}}(\rd F(x_k))^{*}[F(x_k)],
\EE
and the Riemannian inexact Newton point $\Delta x_k^{IN}$.

Based on the above analysis and sparked by the ideas in \cite{PS08,S06,ZBJ18}, we propose the following
Riemannian  inexact Newton dogleg method for solving \eqref{eq:nleq}.
\begin{algorithm}\label{dogleg1}
{\rm (Riemannian inexact Newton dogleg method)}
\begin{description}
\item [{\rm Step 0.}] Choose an initial point $ x_0 \in \cm$, $\epsilon>0$, $0<t<1$, $0<\sigma_{\max}<1$, $0<\theta_{\max}<1$, $0<\delta_{\min}<1$, $\delta_0\ge \delta_{\min}$,
                  and a nonnegative sequences $\{\bar{\eta}_k\in (0,1)\}$ with
                  $\lim_{k\to \infty} \bar{\eta}_k = 0$. Let $k:=0$.

\item [{\rm Step 1.}] If $\|F(x_k)\|<\epsilon$, then stop.
\item [{\rm Step 2.}] Apply the CG method to solve  \eqref{RNEWTONEQ} for $\Delta z_k\in T_{F(x_k)}\ce$
such that
    \BE\label{eq:tol1}
    \begin{array}{l}
    \big\|\big(\rd F(x_k)\circ(\rd F(x_k))^* +\sigma_k \mathrm{id}_{T_{F(x_k)}\mathcal{E}}\big)
    [ \Delta z_k ] + F(x_k)\big\|\leq  \eta_k\|  F(x_k) \|
    \end{array}
    \EE
    and
    \BE\label{eq:tol2}
    \begin{array}{l}
    \| \rd F(x_k)\circ (\rd F(x_k))^* [ \Delta z_k ] + F(x_k)\| < \|  F(x_k) \| ,
    \end{array}
    \EE
    where
    \BE\label{def:skek}
    \sigma_k:=\min\{\sigma_{\max},\| F(x_k)\|\} \quad \mbox{and} \quad \eta_k:=\min\{\bar{\eta}_k,\| F(x_k)\|\}.
    \EE

\item [{\rm Step 3.}]
    Define
    \BE\label{eq:direstep}
    \Delta x_k^{IN}:= (\rd F(x_k))^*[ \Delta z_k ].
    \EE
    Compute $\widehat{\Delta x}_k^{CP}$ by \eqref{CP}.
       Determine $\Delta x_k \in \widehat{\Gamma}^{DL}_k$ with $\min\{\delta_{\min},\|\Delta x_k^{IN}\|\} \leq \|\Delta x_k\| \leq \delta_k$.

\item [{\rm Step 4.}] While ${\rm Ared}_{k}(\Delta x_{k})<t\cdot {\rm {\rm {\rm Pred}}_{k}}(\Delta x_{k})$ do: \\
      If $\delta_k = \delta_{\min}$, stop; else choose $\theta_k \in (0,\theta_{\max}]$. \\
       Update $\delta_k = \max\{ \theta_k \delta_k, \delta_{\min}\}$.\\
      Redetermine $\Delta x_k \in \widehat{\Gamma}^{DL}_k$ with $\min\{\delta_{\min},\|\Delta x_k^{IN}\|\} \leq \|\Delta x_k\| \leq \delta_k$.

\item [\rm {Step 5.}] Set $x_{k+1} := R_{x_k} (\Delta x_k )$.
      Update $\delta_{k+1} \in [\delta_{\min},\infty)$.

\item [{\rm Step 6.}] Replace $k$ by $k + 1$ and go to {\rm Step 1}.
\end{description}
\end{algorithm}

On  Algorithm \ref{dogleg1}, we have several remarks as follows:
\begin{itemize}
\item The choices of $\sigma_k$ and $\eta_k$ in {\rm (\ref{def:skek})} are sparked by the similar idea in {\rm \cite{ZBJ18}}.

\item  The norm of the linear model $F(x_k) + \rd F(x_k) [ \xi_{x_k} ]$ is
monotone decreasing along the segment of  the inexact dogleg curve $\widehat{\Gamma}_{k}^{DL}$ between $0_{x_k}$ and $\widehat{\Delta x}_k^{CP}$, while it may not be monotone decreasing along the segment of $\widehat{\Gamma}_{k}^{DL}$ between $\widehat{\Delta x}_k^{CP}$
and $\Delta x_{k}^{IN}$.

\item We observe from Step 4 of  Algorithm \ref{dogleg1},  \eqref{def:ared}, and \eqref{def:pred} that for all $k\ge 0$,
\BE\label{seq:fxk}
\| F(x_k)\|-\| F(x_{k+1})\|\ge t(\| F(x_k)\|-\| F(x_k)+\rd F(x_k)[\Delta x_k]\|).
\EE
This shows that the sequence $\{F(x_k)\}$ is monotone decreasing if  Algorithm \ref{dogleg1} does not break down.
\item
The procedure for determining $\Delta x_k$ and $\delta_{k+1}$ in Steps 3--5 of Algorithm \ref{dogleg1} is presented  in {\rm Section \ref{sec5}}.
\end{itemize}
\subsection{Convergence analysis}

In this subsection, we establish the global and quadratic convergence of Algorithm \ref{dogleg1}.
Let
\BE\label{LEVELSET}
\Omega: = \big\{x\in \cm \ | \ \|F(x)\| \leq \|F(x_0)\| \big\}.
\EE

To derive the global convergence of Algorithm \ref{dogleg1}, we need the following basic assumption.
\begin{assumption} \label{ASSUM}
\begin{enumerate}
\item The mapping $F:\cm \to \ce$ is continuously differentiable on the level set $\Omega$.

\item  For the retraction $R$ defined on $\cm$, there exist two scalars $\nu >0$ and $\mu_{\nu} >0$ such that
\[
\nu\| \Delta x \| \geq  {\rm dist}\big(x,R_{x} (\Delta x) \big),
\]
for all $x\in \Omega$ and $\Delta x\in T_x\cm$ with $\| \Delta x\|\leq \mu_{\nu}$,
where ``{\rm dist}" means the Riemannian distance on $\cm$.
\end{enumerate}
\end{assumption}
\begin{remark}
If the level set $\Omega$ is compact, then the second condition in   {\rm Assumption \ref{ASSUM}}
is  satisfied. This is guaranteed if the Riemannian manifold $\cm$ is compact {\rm \cite[p. 149]{AMS08}}.
\end{remark}

To show the convergence of Algorithm {\rm\ref{dogleg1}}, for the iterates $\Delta x_k^{IN}$,
$\widehat{\Delta x}_k^{CP}$, and  $\Delta x_k$ generated by Algorithm \ref{dogleg1},
define
\begin{eqnarray}
&&\eta^{IN}_k: = \frac{\|F(x_k)+\rd F(x_k)[\Delta x_k^{IN}]\|}{\|F(x_k)\|},\label{def:etain} \\
&&\eta_k^{CP} :=\frac{\|F(x_k)+\rd F(x_k)[\widehat{\Delta x}_k^{CP}]\|}{\|F(x_k)\|}, \label{def:etacp}\\
&&\tau_k: =\frac{\|F(x_{k})+\rd F(x_{k})[\Delta x_{k}]\|}{\|F(x_{k})\|}\equiv 1-\frac{{\rm Pred}_{k}(\Delta x_k)}{\|F(x_{k})\|}.\label{def:tauk}
\end{eqnarray}

In the following, we give some lemmas, which are necessary for deducing  the global  convergence of Algorithm \ref{dogleg1}. First, by following the similar arguments of \cite[Lemma 1]{ZBJ18}, we have the following result on the reachability of
conditions \eqref{eq:tol1} and \eqref{eq:tol2} for solving  \eqref{RNEWTONEQ}.

\begin{lemma}\label{lem:tol}
Let $x_k$ be the current iterate generated by {\rm Algorithm {\rm\ref{dogleg1}}}. If $\grad f(x_k) \neq 0_{x_k}$,
then we can solve {\rm \eqref{RNEWTONEQ}} sufficiently accurately such that {\rm (\ref{eq:tol1})}
and {\rm (\ref{eq:tol2})} are satisfied.
\end{lemma}

On the quantity $\eta^{IN}_k$ defined by \eqref{def:etain}, we have the following lemma.
\begin{lemma}\label{lem:etain}
Let $x_k$ be the current iterate  generated by {\rm Algorithm \ref{dogleg1}}. If $\grad f(x_k) \neq 0_{x_k}$, then, for $\eta^{IN}_k$ defined by \eqref{def:etain}, we have
\[
\eta^{IN}_k \leq  \frac{\sigma_k}{\sigma_k + \lambda_{\min}\big(\rd F(x_k)\circ (\rd F(x_k))^*\big) } + \eta_k
\quad \mbox{and} \quad \eta^{IN}_k < 1,
\]
where $\lambda_{\min}(\cdot)$ denotes the smallest eigenvalue of a linear operator.
\end{lemma}
\begin{proof}
This follows from Lemma \ref{lem:tol} and \cite[Lemma 3]{ZBJ18}.
\end{proof}

On the quantity $\eta_k^{CP}$  defined by \eqref{def:etacp}, we have the following result.
\begin{lemma}\label{lem:etain2}
Let $x_k$ be the current iterate  generated by {\rm Algorithm {\rm\ref{dogleg1}}}. If $\grad f(x_k) \neq 0_{x_k}$, then, for $\eta_k^{CP}$ defined by \eqref{def:etacp}, we have
\[
\eta_k^{CP} <1.
\]
\end{lemma}
\begin{proof}
By hypothesis, $\grad f(x_k) = (\rd F(x_k))^*[F(x_k)] \neq 0_{x_k}$. Thus,
\begin{eqnarray}\label{bd:dfxk}
0&<&\|(\rd F(x_k))^{*}[F(x_k)]\|^2=
\langle (\rd F(x_k))^{*}[F(x_k)], (\rd F(x_k))^{*}[F(x_k)]\rangle \nonumber\\
&=& \langle F(x_k),\rd F(x_k)\circ(\rd F(x_k))^{*}[F(x_k)]\rangle \nonumber\\
&\le &  \|F(x_k)\|\cdot \|\rd F(x_k)\circ(\rd F(x_k))^{*}[F(x_k)]\|.
\end{eqnarray}
It follows from \eqref{CP}  that
\begin{eqnarray*}
&&\|F(x_k)+\rd F(x_k)[\widehat{\Delta x}_k^{CP}]\|^2 \\
&=& \|F(x_k)\|^2 + \|\rd F(x_k)[\widehat{\Delta x}_k^{CP}]\|^2 + 2\langle F(x_k),\rd F(x_k)[\widehat{\Delta x}_k^{CP}] \rangle \\
&=& \|F(x_k)\|^2 + \frac{\|(\rd F(x_k))^{*}[F(x_k)]\|^4}
{\|\rd F(x_k)\circ(\rd F(x_k))^{*}[F(x_k)]\|^4} \|\rd F(x_k)\circ(\rd F(x_k))^{*}[F(x_k)]\|^2 \\
&& - \frac{2\|(\rd F(x_k))^{*}[F(x_k)]\|^2}
{\|\rd F(x_k)\circ(\rd F(x_k))^{*}[F(x_k)]\|^2} \langle F(x_k),\rd F(x_k)\circ(\rd F(x_k))^{*}[F(x_k)] \rangle \\
&=&  \|F(x_k)\|^2 - \frac{\|(\rd F(x_k))^{*}[F(x_k)]\|^4}
{\|\rd F(x_k)\circ(\rd F(x_k))^{*}[F(x_k)]\|^2} \\
&=& \|F(x_k)\|^2 \Bigg(1- \frac{\|(\rd F(x_k))^{*}[F(x_k)]\|^4}
{\|F(x_k)\|^2\|\rd F(x_k)\circ(\rd F(x_k))^{*}[F(x_k)]\|^2}\Bigg).
\end{eqnarray*}
This, together with \eqref{def:etacp} and \eqref{bd:dfxk}, yields $\eta_k^{CP} <1$.
\end{proof}

On the quantity $\tau_k$ defined by \eqref{def:tauk}, we have the following result.
\begin{lemma}\label{lem:tauk}
Let $x_k$ be the current iterate  generated by {\rm Algorithm \ref{dogleg1}}. If $\grad f(x_k) \neq 0_{x_k}$, then, for $\tau_k$ defined by \eqref{def:tauk}, we have
\[
0\leq\tau_k <1, \quad
\|F(x_{k+1})\| \le (1-t(1-\tau_k))\|F(x_k)\|.
\]
\end{lemma}
\begin{proof}
By hypothesis, $\grad f(x_k) = (\rd F(x_k))^*[F(x_k)] \neq 0_{x_k}$, i.e., $x_k$ is not a stationary point of $f$.
Since  $\|F(x_k) + \rd F(x_k) [ \xi_{x_k} ]\|$ is strictly monotone decreasing along the segment of
$\widehat{\Gamma}_{k}^{DL}$ between $0_{x_k}$ and $\widehat{\Delta x}_k^{CP}$,
if $\Delta x_k$ lies on $\widehat{\Gamma}^{DL}_k$ between $0_{x_k}$ and $\widehat{\Delta x}_k^{CP}$,  we have
\BE\label{DL1231-2}
\eta_{k}^{CP}\|F(X_k)\| \leq \|F(x_k)+\rd  F(x_k) [\Delta x_k]\|
< \|F(x_k)\|.
\EE
If $\Delta x_k$ lies on $\widehat{\Gamma}^{DL}_k$ between
$\widehat{\Delta x}_k^{CP}$ and $\Delta x_{k}^{IN}$, then it follows from
\eqref{def:etain}, (\ref{def:etacp}),  norm convexity,
and Lemmas \ref{lem:etain} and \ref{lem:etain2} that
\BE\label{DL1231}
0\leq \|F(x_k)+\rd  F(x_k) [\Delta x_k]\|\leq\max\{\eta_{k}^{CP},\eta_{k}^{IN}\}\|F(x_k)\|
< \|F(x_k)\|.
\EE
Based on (\ref{def:tauk}), (\ref{DL1231-2}), and (\ref{DL1231}), we can obtain
$0\leq\tau_k <1$.
Then, we have by \eqref{seq:fxk},
\[
\| F(x_{k+1})\| \le \| F(x_k)\|- t(\| F(x_k)\|-\| F(x_k)+\rd F(x_k)[\Delta x_k]\|)=\big(1-t(1-\tau_k)\big) \| F(x_k)\|,
\]
This completes the proof.
\end{proof}


On the iterate $\Delta x_k^{IN}$ generated by Algorithm {\rm\ref{dogleg1}},  we have the following result.
\begin{lemma}\label{lem:dxkin}
Let $x_k$ be the current iterate  generated by {\rm Algorithm \ref{dogleg1}}. If $\grad f(x_k) \neq 0_{x_k}$, then
\[
\|\Delta x_k^{IN}\| \le (1+\eta_k)\normmm{(\rd F(x_k))^\dag}\cdot \|F(x_k)\|.
\]
\end{lemma}
\begin{proof}
It follows from the same arguments of \cite[Lemma 2]{ZBJ18}.
\end{proof}

On the iterate $\widehat{\Delta x}_k^{CP}$  generated by Algorithm {\rm\ref{dogleg1}},  we have the following result.
\begin{lemma}\label{lem:dxkcp}
Let $x_k$ be the current iterate  generated by {\rm Algorithm {\rm\ref{dogleg1}}}. If $\grad f(x_k) \neq 0_{x_k}$
and $\rd F(x_k)$ is surjective, then
\BE\label{DELTACPEST}
\|\widehat{\Delta x}_k^{CP}\| \le\lambda_{\min}^{-\frac{1}{2}}\big(\rd F(x_k)\circ (\rd F(x_k))^*\big) \|F(x_{k})\|.
\EE
\end{lemma}
\begin{proof}
By hypothesis, $\grad f(x_k) = (\rd F(x_k))^*[F(x_k)] \neq 0_{x_k}$. Since $\rd F(x_k)$ is surjective, we know that $\lambda_{\min}(\rd F(x_k)\circ (\rd F(x_k))^*) >0$. Using the definition of $\widehat{\Delta x}_k^{CP}$ we have
\begin{eqnarray*}
\Big\|\widehat{\Delta x}_k^{CP}\Big\|&=&   \frac{\|(\rd F(x_{k}))^{*}[F(x_{k})]\|^{2}}{\|\rd F(x_{k})\circ(\rd F(x_{k}))^{*}[F(x_{k})]\|^{2}}
\|(\rd F(x_{k}))^{*}[F(x_{k})]\| \\
&=&  \frac{\|(\rd F(x_{k}))^{*}[F(x_{k})]\|^{4}}
{\|\rd F(x_{k})\circ(\rd F(x_{k}))^{*}[F(x_{k})]\|^{2}}\cdot
\frac{1}{\|(\rd F(x_{k}))^{*}[F(x_{k})]\|}\\
&=&  \frac{\langle F(x_{k}),\rd F(x_{k})\circ(\rd F(x_{k}))^{*}[F(x_{k})]\rangle^{2}}
{\|\rd F(x_{k})\circ(\rd F(x_{k}))^{*}[F(x_{k})]\|^{2}}\cdot
\frac{1}{\|(\rd F(x_{k}))^{*}[F(x_{k})]\|}\\
&\leq&  \frac{\|F(x_{k})\|^{2}}{\|(\rd F(x_{k}))^{*}[F(x_{k})]\|}
= \left(\frac{\|F(x_{k})\|^2}{\langle (F(x_{k}),\rd F(x_{k})\circ\rd F(x_{k}))^{*}[F(x_{k})]\rangle}\right)^{\frac{1}{2}} \|F(x_{k})\|\\
&\le& \lambda_{\min}^{-\frac{1}{2}}\big(\rd F(x_k)\circ (\rd F(x_k))^*\big) \|F(x_{k})\|.
\end{eqnarray*}
\end{proof}

We now derive the following result on the sequence $\{\eta^{IN}_k\}$ generated by Algorithm {\rm\ref{dogleg1}}, where $\eta^{IN}_k$ is defined by \eqref{def:etain}.
\begin{lemma}\label{lemma23}
Suppose the first condition of {\rm Assumption {\rm \ref{ASSUM}}} is satisfied and {\rm Algorithm \ref{dogleg1}} generates an infinite iterative sequence $\{x_{k}\}$. Let $\bar{x}$ be an accumulation point  of $\{x_k\}$ and $\{x_k\}_{k\in \mathcal{K}}$ be a subsequence of $\{x_k\}$ converging to $\bar{x}$. If $\grad f(\bar{x}) \neq 0_{\bar{x}}$, then, for $\eta^{IN}_k$ defined by \eqref{def:etain}, we have
\[
\lim\limits_{k \to \infty, k\in\ck} \eta_k^{IN} < 1.
\]
\end{lemma}
\begin{proof}
By hypothesis,  $\grad f(\bar{x}) = (\rd F(\bar{x}))^*[F(\bar{x})] \neq 0_{\bar{x}}$. Thus $F(\bar{x}) \neq 0$. Since $\bar{x}$ is an accumulation point  of $\{x_k\}$, there exists a subsequence $\{x_k\}_{k\in\ck}$, which converges to $\bar{x}$. Hence, by the continuous differentiability of $F$, there exists a constant $c>0$ such that for all $k\in\ck$ sufficiently large,
\BE\label{BOUNDBELOWC}
\|F(x_k)\| \geq c.
\EE
This, together with  (\ref{def:skek}), yields
\BE\label{sigmaklbd}
\bar{\sigma}=\lim_{k \to \infty,\; k\in\ck} \sigma_k \geq \min\{ \sigma_{\max}, c\}>0.
\EE
We note that   $F$ is continuously differentiable. Thus,
\BE\label{FDFCC}
\lim\limits_{k \to \infty,k\in \mathcal{K}} \rd F(x_k) = \rd F(\bar{x}), \quad \mbox{and} \quad
\lim\limits_{k \to \infty,k\in \mathcal{K}} (\rd F(x_k))^* = (\rd F(\bar{x}))^*.
\EE
Let
\BE\label{def:VK}
W(x_k) := \big( \rd F(x_k) \circ (\rd F(x_k))^* +\sigma_k \mathrm{id}_{T_{F(x_k)} \mathcal{E}}\big)[ \Delta z_k ]
    + F(x_k).
\EE
By hypothesis,  $\lim_{k\to \infty}\bar{\eta}_k = 0$.  It follows from
(\ref{eq:tol1}), (\ref{def:skek}), and (\ref{def:VK}) that
\BE\label{VKLIMIT}
\lim_{k\to \infty} W(x_k) =  0.
\EE
Using (\ref{eq:tol1}) and (\ref{def:VK}) we have
\BE\label{DELTAZK}
\Delta z_k =  \big(\rd F(x_k) \circ (\rd F(x_k))^* +\sigma_k \mathrm{id}_{T_{F(x_k)} \mathcal{E}}\big)^{-1}
[W(x_k) - F(x_k)].
\EE
From (\ref{eq:direstep}), (\ref{sigmaklbd}), (\ref{FDFCC}),  (\ref{VKLIMIT}), and (\ref{DELTAZK}) we obtain
\begin{eqnarray}\label{FDFEST1}
& &\lim\limits_{k \to \infty,k\in \mathcal{K}}  F(x_k) + \rd F(x_k)[\Delta x^{IN}_k] \nonumber\\
&=&\lim\limits_{k \to \infty,k\in \mathcal{K}}  F(x_k) + \rd F(x_k)\circ (\rd F(x_k))^*[\Delta z_k] \nonumber\\
&=&  F(\bar{x}) - \rd F(\bar{x})\circ (\rd F(\bar{x}))^* \circ
 \Big(\rd F(\bar{x}) \circ (\rd F(\bar{x}))^* +\bar{\sigma} \mathrm{id}_{T_{F(\bar{x})} \mathcal{E}}\Big)^{-1}[ F(\bar{x}) ] \nonumber\\
&=& \bar{\sigma}\cdot \Big(\rd F(\bar{x}) \circ (\rd F(\bar{x}))^*
+\bar{\sigma} \mathrm{id}_{T_{F(\bar{x})} \mathcal{E}}\Big)^{-1} [F(\bar{x})].
\end{eqnarray}
Since $\grad f(\bar{x}) = (\rd F(\bar{x}))^*[F(\bar{x})] \neq 0_{\bar{x}}$, we have
\BE\label{FDFEST2}
F(\bar{x})\not \in {\rm null}((\rd F(\bar{x}))^*).
\EE
Using (\ref{FDFEST1}) and (\ref{FDFEST2}) we have
\BE\label{FDFEST3}
\lim\limits_{k \to \infty,k\in \mathcal{K}} \|F(x_k) + \rd F(x_k)[\Delta x^{IN}_k]\| < \|F(\bar{x})\|.
\EE
From \eqref{def:etain} and (\ref{FDFEST3}), we have
\[
\lim\limits_{k \to \infty, k\in\ck} \eta_k^{IN} < 1.
\]
\end{proof}

On the global convergence of Algorithm \ref{dogleg1}, we have the following theorem.
\begin{theorem}\label{Thm:1}
Suppose the first condition of  {\rm Assumption {\rm \ref{ASSUM}}}  is satisfied  and $\{x_{k}\}$ is an infinite sequence generated by {\rm Algorithm \ref{dogleg1}}. Then every accumulation point of $\{x_k\}$ is a stationary point of $f$.
\end{theorem}

\begin{proof}
Let $\bar{x}$ be an accumulation point of $\{x_{k}\}$, then there exists a subsequence $\{x_k\}_{k\in \mathcal{K}}$ of $\{x_k\}$
such that $\lim_{k\to \infty, k\in \mathcal{K}} x_k = \bar{x}$.
%
By contradiction, we assume  that $\bar{x}$ is not a stationary point of $F$. Then we have $\grad f(\bar{x}) = (\rd F(\bar{x}))^*[F(\bar{x})] \neq 0_{\bar{x}}$ and thus $F(\bar{x})\neq 0$.
Since $F$ is continuously differentiable, we have
\BE\label{DFFNORMBD}
0< \inf_{k\in \mathcal{K}} \|F(x_k)\| \quad \mbox{and} \quad
0< \inf_{k\in \mathcal{K}} \normmm{\rd F(x_{k})}  \le  \sup_{k\in \mathcal{K}} \normmm{\rd F(x_{k})} < \infty.
\EE
Using the continuous differentiability of $F$, (\ref{CP}),  \eqref{def:etacp}, and Lemma \ref{lem:etain2} we can obtain
\BE\label{CPLIMIT}
\lim\limits_{k\to \infty, k\in \mathcal{K}}  \widehat{\Delta x}_k^{CP} =-\frac{\|(\rd F(\bar{x}))^{*}[F(\bar{x})]\|^{2}}
{\|\rd F(\bar{x})\circ (\rd F(\bar{x}))^{*}[F(\bar{x})]\|^{2}}(\rd F(\bar{x}))^{*}[F(\bar{x})]
= \widehat{\Delta \bar{x}}^{CP},
\EE
and
\BE\label{CPLIMITeta}
\lim\limits_{k\to \infty, k\in \mathcal{K}} \eta_k^{CP}  =
\frac{\|F(\bar{x})+\rd F(\bar{x})[\widehat{\Delta \bar{x}}^{CP}]\|}{\|F(\bar{x})\|} <1.
\EE
By (\ref{CPLIMIT}) and (\ref{CPLIMITeta}), there exist two constants $\kappa_1>0$ and $\eta_{\max}^{CP}\in (0,1)$ such that for all $k\in \mathcal{K}$ sufficiently large,
\BE\label{CPPCBOUND1}
\| \widehat{\Delta x}_k^{CP}\| \leq \kappa_1 \quad \mbox{and} \quad \eta_k^{CP} \leq \eta_{\max}^{CP} < 1.
\EE

By assumption,  $\grad f(\bar{x}) \neq 0_{\bar{x}}$. By Lemma \ref{lemma23},  there exists a constant $\eta_{\max}^{IN}\in (0,1)$ such that
for all $k\in \mathcal{K}$ sufficiently large,
\BE\label{CPPCBOUND2}
 \eta_k^{IN} \leq \eta_{\max}^{IN} < 1.
\EE
Using \eqref{def:etain}, (\ref{CPPCBOUND2}), and triangle inequality we have for all $k\in \mathcal{K}$ sufficiently large,
\[
\|F(x_k)\|-\|\rd  F(x_k) [\Delta x_{k}^{IN}]\|\leq\|F(x_k)+\rd  F(x_k) [\Delta x_{k}^{IN}]\| = \eta_{k}^{IN}\|F(x_k)\|\leq\eta^{IN}_{\max}\|F(x_k)\|,
\]
which, together with (\ref{DFFNORMBD}), implies that for all $k\in \mathcal{K}$ sufficiently large,
\BE\label{DELTAIINNEWTON}
\|\Delta x_{k}^{IN}\|\geq\frac{1-\eta^{IN}_{\max}}{\normmm{\rd F(x_{k})}}\|F(x_k)\|
\ge \frac{1-\eta^{IN}_{\max}}{\sup\limits_{k\in \mathcal{K}}\normmm{\rd F(x_{k})}} \inf_{k\in \mathcal{K}}\|F(x_k)\|\ge \overline{\delta},
\EE
where  $\overline{\delta}>0$ is  a constant.

If $\Delta x_k$ lies on $\widehat{\Gamma}^{DL}_k$ between
$\widehat{\Delta x}_k^{CP}$ and $\Delta x_{k}^{IN}$, then it follows from
\eqref{def:etain}, (\ref{def:etacp}), (\ref{CPPCBOUND1}), (\ref{CPPCBOUND2}), and norm convexity that for all $k\in \mathcal{K}$ sufficiently large,
\BE\label{DL1}
\|F(x_k)+\rd  F(x_k) [\Delta x_k]\|\leq\max\{\eta_{k}^{CP},\eta_{k}^{IN}\}\|F(x_k)\|
\leq\max\left\{\eta_{\max}^{CP},\eta_{\max}^{IN}\right\}\|F(x_k)\|.
\EE
If $\Delta x_k$ lies on $\widehat{\Gamma}^{DL}_k$ between $0_{x_k}$ and $\widehat{\Delta x}_k^{CP}$, then we have  by \eqref{DELTAIINNEWTON}, for all $k\in \mathcal{K}$ sufficiently large,
\BE\label{deltalowbd2}
0<\delta_*:=\min\{\delta_{\min},\bar{\delta} \}\leq\|\Delta x_k\|\leq \|\widehat{\Delta x}_k^{CP}\|.
\EE
We also  note that  the norm of the local linear model (\ref{llm}) is monotone decreasing along the segment of $\widehat{\Gamma}^{DL}_k$
between $0_{x_k}$ and $\widehat{\Delta x}_k^{CP}$. Using (\ref{CPPCBOUND1}), (\ref{deltalowbd2}), and norm convexity, for $\Delta x_k$ lying on $\widehat{\Gamma}^{DL}_k$ between $0_{x_k}$ and $\widehat{\Delta x}_k^{CP}$, we have for all $k\in \mathcal{K}$ sufficiently large,
\begin{eqnarray}\label{DL2}
&&\|F(x_{k})+\rd F(x_{k})[\Delta x_{k}]\|\le
\left\|F(x_{k})+\rd F(x_{k})\left[\frac{\delta_*}{\|\widehat{\Delta x}_k^{CP}\|}\widehat{\Delta x}_k^{CP}\right]\right\|  \nonumber \\
&\leq&\Big(1-\frac{\delta_*}{\|\widehat{\Delta x}_k^{CP}\|}\Big)\|F(x_{k})\|+\frac{\delta_*}{\|\widehat{\Delta x}_k^{CP}\|}
\|F(x_{k})+\rd F(x_{k})[\widehat{\Delta x}_k^{CP}]\| \nonumber \\
&=&\Big(1-\frac{\delta_*}{\|\widehat{\Delta x}_k^{CP}\|}(1-\eta_{k}^{CP})\Big)\|F(x_{k})\|
\le\Big(1-\frac{\delta_*}{\kappa_1}(1-\eta_{\max}^{CP})\Big)\|F(x_{k})\|.
\end{eqnarray}

From (\ref{DL1}) and (\ref{DL2}) we have for all $k\in \mathcal{K}$ sufficiently large,
\[
\|F(x_{k})+\rd F(x_{k})[\Delta x_{k}]\|\leq\overline{\eta}\|F(x_{k})\|,
\]
where
\[
\overline{\eta}:=\max\left\{\eta_{\max}^{CP},\eta_{\max}^{IN}, 1-\frac{\delta_*}{\kappa_1}(1-\eta_{\max}^{CP}) \right\}.
\]
Thus for all $k\in \mathcal{K}$ sufficiently large,
\[
\frac{{\rm Pred}_{k}(\Delta x_k)}{\|F(x_{k})\|}=\frac{\|F(x_{k})\|-\|F(x_{k})+\rd F(x_{k})[\Delta x_{k}]\|}{\|F(x_{k})\|}
\geq (1-\overline{\eta})>0, \quad \forall k\in \mathcal{K},\; k > \tilde{k}.
\]
This implies that the series $\sum_{k=0}^{\infty}\frac{{\rm Pred}_{k}(\Delta x_k)}{\|F(x_{k})\|}$ diverges. This, together with  \eqref{def:tauk}, means that $\sum_{k=0}^{\infty}(1-\tau_k)$ diverges. It follows from Lemma \ref{lem:tauk} that
\begin{eqnarray}\label{ineq:fxk1}
\|F(x_{k+1})\| &\le& \big(1-t(1-\tau_k)\big)\|F(x_k)\| \le \prod_{l=0}^k\big(1-t(1-\tau_l)\big)\|F(x_0)\| \nonumber \\
&\le & \exp\Big(-t\sum_{l=0}^k(1-\tau_l)\Big)\|F(x_0)\| \to 0,\quad \mbox{as $k\to\infty$}.
\end{eqnarray}
By the assumption that $F$ is continuously differentiable we have $F(\bar{x})=0$, which is a contradiction. The proof is complete.
\end{proof}

To show the convergence of the sequence  $\{x_{k}\}$ generated by Algorithm \ref{dogleg1}, we need the following lemma.
\begin{lemma}\label{lemma231}
Suppose the first condition of  {\rm Assumption {\rm \ref{ASSUM}}} is satisfied and $\{x_{k}\}$ is an infinite sequence generated by {\rm Algorithm \ref{dogleg1}}. Let $\bar{x}$ be an accumulation point  of $\{x_k\}$
and $\{x_k\}_{k\in \mathcal{K}}$ be a subsequence of $\{x_k\}$ converging to $\bar{x}$.
If $\rd F(\bar{x})$ is surjective, then
\[
\lim\limits_{k \to \infty,k\in \mathcal{K}} \eta_k^{IN} = 0.
\]
\end{lemma}

\begin{proof}
By hypothesis, $\bar{x}$ is an accumulation point of the sequence $\{x_{k}\}$ generated by {\rm Algorithm {\rm \ref{dogleg1}}}. It follows from Theorem \ref{Thm:1} that $\bar{x}$ is a stationary point of $f$, i.e.,  $\grad f(\bar{x}) = (\rd F(\bar{x}))^*[F(\bar{x})] \neq 0_{\bar{x}}$.
Since $\rd F(\bar{x})$ is surjective, we have  $F(\bar{x})=0$.
By the monotonicity of $\{\|F(x_k)\|\}$ and $\lim\limits_{k\to \infty, k\in \mathcal{K}} x_k = \bar{x}$ we have
\BE\label{FXZERONM}
\lim_{k\to \infty, k\in \mathcal{K}} \|F(x_k)\| = 0.
\EE
From  (\ref{def:skek}) and (\ref{FXZERONM}) we  obtain
\BE\label{T222}
\lim_{k\to \infty,k\in\ck} \sigma_k  =0= \lim_{k\to \infty,k\in\ck} \eta_k.
\EE
By hypothesis, $\rd  F(\bar{x})$ is surjective and $F$ is continuously differentiable. Thus,
\BE\label{T223}
\lim_{k\to \infty, k\in \mathcal{K}} \lambda_{\min}\big(\rd F(x_k)\circ (\rd F(x_k))^*\big)
= \lambda_{\min}\big(\rd F(\bar{x})\circ (\rd F(\bar{x}))^*\big) >0.
\EE
It follows from Lemma \ref{lem:etain},  (\ref{T222}), and (\ref{T223}) that
$\lim\limits_{k \to \infty,k\in \mathcal{K}} \eta_k^{IN} = 0$.
\end{proof}

On the convergence of the sequence  $\{\|F(x_k)\|\}$ generated by Algorithm \ref{dogleg1}, we have the following result.
\begin{theorem}\label{thm:stauk}
Suppose {\rm Assumption {\rm \ref{ASSUM}}}   is satisfied  and $\{x_{k}\}$ is an infinite sequence generated by {\rm Algorithm \ref{dogleg1}}.
If $\bar{x}$ is an accumulation point of $\{x_{k}\}$ such that $\rd F(\bar{x})$ is surjective,
then $\sum\limits_{k=0}^{\infty}(1-\tau_k)$ diverges and $\lim\limits_{k\to\infty}\|F(x_k)\|=0$.
\end{theorem}

\begin{proof}
By Theorem \ref{Thm:1}, $\bar{x}$ is a stationary point of $f$. Thus, $\grad f(\bar{x}) = (\rd F(\bar{x}))^*[F(\bar{x})] = 0_{\bar{x}}$.  Since $\rd F(\bar{x})$ is surjective, we have  $F(\bar{x})=0$.
Let $\{x_k\}_{k\in \mathcal{K}}$ be a subsequence of $\{x_k\}$ converging to $\bar{x}$, i..e,
$\lim\limits_{k\to \infty, k\in \mathcal{K}} x_k = \bar{x}$.
By hypothesis,  $F$ is continuously differentiable and $\rd F(\bar{x})$  is surjective. Hence, there exists a constants $\kappa_2>0$ such that for all $k\in \mathcal{K}$ sufficiently large,
\BE\label{GINVERSE2}
\normmm{\rd F(x_{k})} \leq \kappa_2
\quad\mbox{and}\quad
 \lambda_{\min}\big(\rd F(x_k)\circ (\rd F(x_k))^*\big)\ge\frac{1}{2}\bar{\lambda}_{\min},
\EE
where $\kappa_2 \ge \sqrt{\frac{1}{2}\bar{\lambda}_{\min}}$
and $\bar{\lambda}_{\min}:=\lambda_{\min}(\rd F(\bar{x})\circ (\rd F(\bar{x}))^*)>0$.
From (\ref{DELTACPEST}) and (\ref{GINVERSE2}), we have for all $k\in \mathcal{K}$ sufficiently large,
\begin{eqnarray} \label{CP1}
\|\widehat{\Delta x}_k^{CP}\|
&\leq& \lambda_{\min}^{-\frac{1}{2}}\big(\rd F(x_k)\circ (\rd F(x_k))^*\big) \|F(x_k)\|
\leq  \frac{\|F(x_k)\|}{\sqrt{\frac{1}{2}\bar{\lambda}_{\min}}}
\leq  \frac{\|F(X_0)\|}{\sqrt{\frac{1}{2}\bar{\lambda}_{\min}}}.
\end{eqnarray}
By the definition of $\widehat{\Delta x}_k^{CP}$ in (\ref{CP}) we have for all $k\in \mathcal{K}$ sufficiently large,
\begin{eqnarray}\label{FDFNORMBD}
\|F(x_{k})+\rd F(x_{k})[\widehat{\Delta x}_k^{CP}]\|^2
&=& \|F(x_k)\|^2 - \frac{\|(\rd F(x_{k}))^{*}[F(x_{k})]\|^{4}}
{\|\rd F(x_{k})\circ(\rd F(x_{k}))^{*}[F(x_{k})]\|^{2}}\nonumber \\
&\leq&\|F(x_k)\|^2 -\frac{\|(\rd F(x_{k}))^{*}[F(x_{k})]\|^{2}}
{\normmm{\rd F(x_{k})}^2}.
\end{eqnarray}
In addition, it follows from \eqref{def:etacp}, (\ref{GINVERSE2}), and (\ref{FDFNORMBD}) that for all $k\in \mathcal{K}$ sufficiently large,
\begin{eqnarray}\label{etakkcp22}
\eta_{k}^{CP} &\leq& \ \sqrt{1-\frac{\|(\rd F(x_{k}))^{*}[F(x_{k})]\|^{2}}{\|F(x_{k})\|^{2}\normmm{\rd F(x_{k})}^{2}}}
\le  \sqrt{1-\frac{\langle F(x_{k}), \rd F(x_{k})\circ (\rd F(x_{k}))^{*}[F(x_{k})]\rangle}{\|F(x_{k})\|^{2}\normmm{\rd F(x_{k})}^{2}}} \nonumber \\
&\leq&   \sqrt{1 - \frac{\frac{1}{2}\bar{\lambda}_{\min}}{\kappa_2^2}}\equiv \eta_{\max}^{CP} < 1.
\end{eqnarray}

Using Lemma \ref{lemma231}, there exists a constant $0<\eta^{IN}_{\max}<1$ such that the first inequality of  (\ref{DELTAIINNEWTON}) holds for all $k\in \mathcal{K}$ sufficiently large.
This, together with (\ref{GINVERSE2}) and (\ref{CP1}), implies that  for all $k\in \mathcal{K}$ sufficiently large,
\BE\label{deltakckinp}
\frac{\|\Delta x_k^{IN}\|}{\|\widehat{\Delta x}_k^{CP}\|}
\geq\frac{\Big(1-\eta^{IN}_{\max}\Big)\sqrt{\frac{1}{2}\bar{\lambda}_{\min}}}{\normmm{\rd F(x_{k})}}
\geq  \frac{\Big(1-\eta^{IN}_{\max}\Big)\sqrt{\frac{1}{2}\bar{\lambda}_{\min}}}{\kappa_2}>0.
\EE
By (\ref{CP1}) and (\ref{deltakckinp}) we can obtain  for all $k\in \mathcal{K}$ sufficiently large,
\begin{eqnarray}\label{etacpllbd}
\frac{\min\{\delta_{\min},\|\Delta x_k^{IN}\|\}}{\|\widehat{\Delta x}_k^{CP}\|}
&=& \min\left\{\frac{\delta_{\min}}{\|\widehat{\Delta x}_k^{CP}\|},\frac{\|\Delta x_k^{IN}\|}{\|\widehat{\Delta x}_k^{CP}\|} \right\} \nonumber \\
&\ge&  \min\left\{\frac{\delta_{\min}\sqrt{\frac{1}{2}\bar{\lambda}_{\min}}}{\|F(X_0)\|},\frac{(1-\eta^{IN}_{\max})\sqrt{\frac{1}{2}\bar{\lambda}_{\min}}}{\kappa_2}\right\}
\ge \hat{\delta},
\end{eqnarray}
where  $\hat{\delta}\in (0,1)$ is a constant.

If $\Delta x_{k}$ lies on $\widehat{\Gamma}^{DL}_k$ between $\widehat{\Delta x}_k^{CP}$ and $\Delta x_{k}^{IN}$,
then there exists a constant $0<\eta_{\max}^{CP}<1$ such that  (\ref{DL1}) holds for all $k\in \mathcal{K}$ sufficiently large.
If $\Delta x_{k}$ lies on $\widehat{\Gamma}^{DL}_k$ between $0_{x_k}$ and $\widehat{\Delta x}_k^{CP}$,
then $\min\{\delta_{\min},\|\Delta x_k^{IN}\|\}\leq\|\Delta x_k\|\leq\|\widehat{\Delta x}_k^{CP}\|$.
We note that  the norm of the local linear model (\ref{llm}) is monotone decreasing along the segment of $\widehat{\Gamma}^{DL}_k$
between $0_{x_k}$ and $\widehat{\Delta x}_k^{CP}$. Then, for $\Delta x_{k}$ lying on $\widehat{\Gamma}^{DL}_k$ between $0_{x_k}$
and $\widehat{\Delta x}_k^{CP}$, it follows from norm convexity, (\ref{etakkcp22}) and (\ref{etacpllbd})
that for all $k\in \mathcal{K}$ sufficiently large,
\begin{eqnarray}\label{CPDESCEST}
& &\|F(x_{k})+\rd F(x_{k})[\Delta x_{k}]\| \nonumber \\
&\leq& \left\|F(x_{k})+\rd F(x_{k})\left[\frac{\min\{\delta_{\min},\|\Delta x_k^{IN}\|\}}
{\|\widehat{\Delta x}_k^{CP}\|}\widehat{\Delta x}_k^{CP}\right]\right\|   \nonumber \\
&\leq&\Big(1-\frac{\min\{\delta_{\min},\|\Delta x_k^{IN}\|\}}{\|\widehat{\Delta x}_k^{CP}\|}\Big)\|F(x_{k})\|
 + \frac{\min\{\delta_{\min},\|\Delta x_k^{IN}\|\}}{\|\widehat{\Delta x}_k^{CP}\|}\|F(x_{k})+\rd F(x_{k})[\widehat{\Delta x}_k^{CP}]\| \nonumber \\
&\leq&\Big(1-\min\Big\{\frac{\delta_{\min}}{\|\widehat{\Delta x}_k^{CP}\|},\frac{\|\Delta x_k^{IN}\|}
{\|\widehat{\Delta x}_k^{CP}\|}\Big\}(1-\eta_{\max}^{CP})\Big)\|F(x_{k})\| \nonumber \\
&\leq& \big(1-\hat{\delta}(1-\eta_{\max}^{CP})\big)\|F(x_{k})\|.
\end{eqnarray}
From (\ref{DL1}) and (\ref{CPDESCEST}) we obtain for all $k\in \mathcal{K}$ sufficiently large,
\[
\|F(x_{k})+\rd F(x_{k})[\Delta x_{k}]\|\leq\hat{\eta}\|F(x_{k})\|,
\]
where
\[
\hat{\eta}:=\max\left\{\eta_{\max}^{CP},\eta_{\max}^{IN}, 1- \hat{\delta}(1-\eta_{\max}^{CP}) \right\}.
\]
Therefore, for all $k\in \mathcal{K}$ sufficiently large,
\[
\frac{{\rm Pred}_{k}(\Delta x_k)}{\|F(x_{k})\|}=\frac{\|F(x_{k})\|-\|F(x_{k})+\rd F(x_{k})[\Delta x_{k}]\|}{\|F(x_{k})\|}
\geq (1-\hat{\eta})>0.
\]
This implies that $\sum_{k=0}^{\infty}\frac{{\rm Pred}_{k}(\Delta x_k)}{\|F(x_{k})\|}$ diverges.
This, together with  \eqref{def:tauk}, implies that $\sum_{k=0}^{\infty}(1-\tau_k)$ diverges.  It follows from Lemma \ref{lem:tauk} that \eqref{ineq:fxk1} holds and thus
$\lim_{k\to\infty}\|F(x_k)\|=0$. By using the  continuous differentiability of $F$ we have  $F(\bar{x})=0$. This completes the proof.
\end{proof}

To establish the convergence of the sequence $\{x_k\}$ generated by Algorithm  \ref{dogleg1}, we need the following assumption.
\begin{assumption} \label{ass:nons}
Suppose {\rm Algorithm {\rm \ref{dogleg1}}} does not break down and $\rd F(\bar{x}):T_{\bar{x}}\cm\to T_{F(\bar{x})}\ce$ is surjective,
where $\bar{x}\in\cm$ is an accumulation point of the sequence $\{x_k\}$ generated by {\rm Algorithm {\rm \ref{dogleg1}}}.
\end{assumption}

We note that the iterate $\Delta x_{k}$ lies on $\widehat{\Gamma}_{k}^{DL}$. Thus,
\[
\|\Delta x_{k}\|\le\max\{ \|\Delta x_k^{IN}\|,\|\widehat{\Delta x}_k^{CP}\|  \}.
\]

Based on Lemma \ref{lem:dxkin}, Lemma \ref{lem:dxkcp}, and Theorem \ref{thm:stauk}, following the similar proof of \cite[Theorem 2]{ZBJ18}, we have the following convergence result on Algorithm \ref{dogleg1}.
\begin{theorem}\label{thm:cov}
Suppose  Assumption {\rm \ref{ASSUM}} and Assumption {\rm \ref{ass:nons}} are satisfied. Let $\bar{x}\in\cm$  be an accumulation point of the sequence $\{x_k\}$ generated by Algorithm  {\rm \ref{dogleg1}}. Then the sequence $\{x_k\}$ converges to $\bar{x}$ and $F(\bar{x})=0$.
\end{theorem}

Similar to the proof of \cite[Lemmas 4 and 5]{ZBJ18}, we have the following result on the procedure for determining $\Delta x_k$ in  Algorithm \ref{dogleg1}.
\begin{lemma}\label{lem:stepsize}
Suppose {\rm  Assumption \ref{ASSUM}} and {\rm Assumption \ref{ass:nons}} are satisfied.
Let $\bar{x}\in\cm$  be an accumulation point of the sequence $\{x_k\}$ generated by  {\rm  Algorithm \ref{dogleg1}}.
Then $\lim_{k\to\infty}\| \Delta x_k^{IN} \| = 0$ and $\Delta x_k^{IN}$ satisfies
the Ared/Pred condition {\rm (\ref{Ared/Predcondition})}
for all $k$ sufficiently large.
\end{lemma}

\begin{proof}
By assumption,  Assumptions {\rm \ref{ASSUM}} and {\rm \ref{ass:nons}} are satisfied. By Theorem \ref{thm:cov}, we know that $\lim_{k\to\infty}x_k = \bar{x}$ and $F(\bar{x})=0$.  By hypothesis,  $F$ is continuously differentiable and  $\mathrm{D}F(\bar{x})$ is surjective.  Then for all $k$ sufficiently large,
$\mathrm{D} F(x_k)$ is surjective and
\BE\label{est:df-eig}
\normmm{( \rd F(x_k))^\dag} \leq 2 \normmm{( \rd F(\bar{x}))^\dag}
\quad\mbox{and}\quad
 \lambda_{\min}\big(\rd F(x_k)\circ (\rd F(x_k))^*\big)\ge\frac{1}{2}\bar{\lambda}_{\min},
\EE
where $\bar{\lambda}_{\min}:=\lambda_{\min}(\rd F(\bar{x})\circ (\rd F(\bar{x}))^*)>0$.
By Lemma \ref{lem:dxkin} we have for all $k$ sufficiently large,
\begin{eqnarray}\label{bd:dxin}
\|\Delta x_k^{IN}\| &\le& (1+\eta_k)\normmm{(\rd F(x_k))^\dag}\cdot \|F(x_k)\| \nonumber \\
&\le& (1+\overline{\eta}_k) \normmm{(\rd F(x_k))^\dag}\cdot\|F(x_k)\|  \nonumber \\
&<& 4 \normmm{(\rd F(\bar{x}))^\dag}\cdot\|F(x_k)\|.
\end{eqnarray}
This, together with $\lim_{k\to\infty}\|F(x_k)\|=F(\bar{x})=0$, yields
\[
\lim_{k\to\infty}\| \Delta x_k^{IN} \| = 0.
\]

By hypothesis,  $F$ is continuously differentiable. Then, for all $k$ sufficiently large,
\[
\|F\big(R_{x_k}(\Delta x_k^{IN})\big) - F(x_k) - \mathrm{D} F(x_k) [\Delta x_k^{IN} ] \| \leq
\epsilon_k\| \Delta x_k^{IN} \|,
\]
i.e.,
\BE\label{eq:lips22}
\|\widehat{F}_{x_k}( \Delta x_k^{IN})-\widehat{F}_{x_k}(0_{x_k})- \mathrm{D}\widehat{F}_{x_k}(0_{x_k})[ \Delta x_k^{IN}] \| \leq \epsilon_k\|  \Delta x_k^{IN} \|,
\EE
where $\epsilon_k:= ((1-t)(1-\eta_k^{IN}))/( 4\normmm{ (\mathrm{D} F(x_k))^\dag})$.

Using  \eqref{def:etain} we have for all $k$ sufficiently large,
\begin{eqnarray}\label{GLINESEARCH1}
& &\|F(x_k) + \mathrm{D} F(x_k) [\Delta x_k^{IN}] \| = \eta_k^{IN} \| F(x_k) \|.
\end{eqnarray}
From  \eqref{bd:dxin}, (\ref{eq:lips22}), and \eqref{GLINESEARCH1} we have  for all $k$ sufficiently large,
\begin{eqnarray}\label{GLINESEARCH2}
& &\| F\big(R_{x_k}(\Delta x_k^{IN})\big)  \| = \| \widehat{F}_{x_k}(\Delta x_k^{IN}) \| \nonumber\\
&\leq& \| \widehat{F}_{x_k}(0_{x_k}) + \mathrm{D} \widehat{F}_{x_k}(0_{x_k}) [\Delta x_k^{IN}] \|
 + \|  \widehat{F}_{x_k}(\Delta  x_k^{IN}) - \widehat{F}_{x_k}(0_{x_k}) - \mathrm{D} \widehat{F}_{x_k}(0_{x_k}) [\Delta x_k^{IN}]\|
\nonumber\\
&\leq&  \eta_k^{IN} \|\widehat{F}_{x_k}(0_{x_k}) \| +\epsilon_k \| \Delta x_k^{IN}  \| \nonumber \\
&\leq&  \eta_k^{IN}  \| F(x_k) \| +4\epsilon_k \normmm{(\rd F(\bar{x}))^\dag}\cdot\|F(x_k)\|  \nonumber\\
&\leq&  \Big(\eta_k^{IN} +4\epsilon_k \normmm{(\rd F(\bar{x}))^\dag}\Big) \|F(x_k)\|  \nonumber\\
&=&  \Big( \eta_k^{IN}   + 4\frac{(1-t)(1-\eta_k^{IN})}{ 4\normmm{(\mathrm{D} F(x_k))^\dag}}
\normmm{(\mathrm{D} F(x_k))^\dag}\Big)\|F(x_k)\| \nonumber\\
&=& \big( \eta_k^{IN}  +  (1-t)(1-\eta_k^{IN}) \big) \|F(x_k)\| \nonumber\\
&=& \big( 1- t(1-\eta_k^{IN}) \big) \|F(x_k)\|.
\end{eqnarray}
Using \eqref{GLINESEARCH1} and  \eqref{GLINESEARCH2} we have
\begin{eqnarray*}
& &\|F(x_k)\| -\| F\big(R_{x_k}(\Delta x_k^{IN})\big)  \| \ge t(1-\eta_k^{IN}) \|F(x_k)\| \\
&=& t (\|F(x_k)\|-\eta_k^{IN} \|F(x_k)\|) \\
&=& t (\|F(x_k)\|-\|F(x_k) + \mathrm{D} F(x_k) [\Delta x_k^{IN}] \| ),
\end{eqnarray*}
which implies
\[
\frac{{\rm Ared}_{k}(\Delta x_k^{IN})}{{\rm Pred}_{k}(\Delta x_k^{IN})}
= \frac{\| F(x_k)\|-\| F(R_{x_k}(\Delta x_k^{IN}))\| }{ \| F(x_k)\|-\| F(x_k)+\rd F(x_k)[\Delta x_k^{IN}]\| } \geq t.
\]
The proof is complete.
\end{proof}

Finally, on the  quadratic convergence of Algorithm \ref{dogleg1}, we have the following result. This follows from the similar proof of \cite[Theorem 3]{ZBJ18} by using Lemma \ref{lem:stepsize}. Here, we give the proof for the sake of completeness.
\begin{theorem} \label{th:qc}
Suppose {\rm  Assumptions \ref{ASSUM}} and {\rm \ref{ass:nons}} are satisfied,
and $\Delta x_k = \Delta x_k^{IN}$ for all $k$ sufficiently large.
Let $\bar{x}\in\cm$  be an accumulation point of the sequence $\{x_k\}$ generated by  {\rm Algorithm \ref{dogleg1}}.
Then the sequence $\{x_k\}$ converges to $\bar{x}$ quadratically.

\end{theorem}
\begin{proof}
Since  Assumptions {\rm \ref{ASSUM}} and {\rm \ref{ass:nons}} are satisfied, it follows from Theorem \ref{thm:cov} and Lemma \ref{lem:stepsize} that $\lim_{k\to\infty}x_k = \bar{x}$, $F(\bar{x})=0$, $\Delta x_k = \Delta x_k^{IN}$  for all $k$ sufficiently large, and
\[
\lim_{k\to\infty}\|\Delta x_k\| =\lim_{k\to\infty} \|\Delta x_k^{IN}\| =0.
\]
Moreover, $\mathrm{D}F(\bar{x})$ is surjective and for all $k$ sufficiently large, $\mathrm{D} F(x_k)$ is surjective with  \eqref{est:df-eig} being satisfied.
By using the continuous differentiability of $F$,  there exist two constants $L_1,L_2>0$ such that for all $k$ sufficiently large,
\BE\label{eq:lip4}
\left\{
\begin{array}{l}
\| F(x_k) \|=\| F(X_k) - F(\bar{x}) \| \leq L_1 \mbox{dist}(x_k, \bar{x}), \\[2mm]
\| \widehat{F}_{x_k}(\Delta x_k) - \widehat{F}_{x_k}(0_{x_k}) - \mathrm{D} \widehat{F}_{x_k}(0_{x_k}) [\Delta x_k] \| \leq L_2 \| \Delta x_k\|^2,\\[2mm]
\mbox{dist}\big(x_k, R_{x_k}(\Delta x_k) \big)\le\nu\|\Delta x_k\|,
\end{array}
\right.
\EE
where $\nu$ is the constant given in Assumption {\rm \ref{ASSUM}}. From Lemma \ref{lem:etain}, \eqref{def:skek},
\eqref{GLINESEARCH1}, and (\ref{eq:lip4}), we have  for all $k$ sufficiently large,
\begin{eqnarray}\label{upbd:etakin}
\eta^{IN}_k &\leq&  \frac{\sigma_k}{\sigma_k + \lambda_{\min}\big(\rd F(x_k)\circ (\rd F(x_k))^*\big) } + \eta_k \nonumber\\
&\leq& \displaystyle \frac{1}{\frac{1}{2}\overline{\la}_{\min}+\sigma_k}\sigma_k + \eta_k
\le \displaystyle \frac{2}{\overline{\la}_{\min}}\|F(x_{k})\| + \|F(x_{k})\| \nonumber\\
&\leq& \displaystyle \frac{2+\overline{\la}_{\min}}{\overline{\la}_{\min}}L_1\mbox{dist}(x_k, \bar{x})
\equiv c_1 \dist(x_k, \bar{x}),
\end{eqnarray}
where $c_1 := (L_1(2+\overline{\la}_{\min}))/ \overline{\la}_{\min}$.

Using  \eqref{bd:dxin}, \eqref{GLINESEARCH1}, (\ref{eq:lip4}),
and \eqref{upbd:etakin}, we have  for all $k$ sufficiently large,
\begin{eqnarray}\label{upbd:fxk}
& &\| F(x_{k+1})  \| = \| \widehat{F}_{x_k}(\Delta x_k) \| \nonumber\\
&\leq& \| \widehat{F}_{x_k}(0_{x_k}) + \mathrm{D} \widehat{F}_{x_k}(0_{x_k}) [\Delta x_k] \|
 + \|  \widehat{F}_{x_k}(\Delta  x_k) - \widehat{F}_{x_k}(0_{x_k}) - \mathrm{D} \widehat{F}_{x_k}(0_{x_k}) [\Delta x_k]\|
\nonumber\\
&\leq&  \eta_k^{IN} \| F(x_k) \| +L_2 \| \Delta x_k  \|^2 \nonumber \\
&\leq&  c_1L_1  \dist^2(x_k, \bar{x}) + 16 L_2\normmm{(\rd F(\bar{x}))^\dag}^2\cdot\|F(x_k)\|^2  \nonumber\\
&\leq&  \big(c_1L_1+16 L_1^2L_2\normmm{(\rd F(\bar{x}))^\dag}^2\big) \cdot  \dist^2(x_k, \bar{x})\nonumber\\
&\equiv& c_2\dist^2(x_k, \bar{x}),
\end{eqnarray}
where $c_2:=c_1L_1+16 L_1^2L_2\normmm{(\rd F(\bar{x}))^\dag}^2$.
If follows from \eqref{upbd:etakin} that there exists a constant $\eta_{\max}\in (0,1)$ such that for all $k$ sufficiently large,
\BE\label{etaset3}
\eta_k^{IN} \leq \eta_{\max}.
\EE
From \eqref{bd:dxin}, \eqref{GLINESEARCH2}, \eqref{upbd:fxk}, and \eqref{etaset3},
we have  for all $k$ sufficiently large,
\begin{eqnarray*}
\dist(x_{k+1}, \bar{x}) &\le & \sum^{\infty}_{j=k+1} \mbox{dist} (x_j, x_{j+1})
= \sum^{\infty}_{j=k+1} \mbox{dist} \big(x_j, R_{x_j}( \Delta x_j) \big)\\
&\leq&  \sum^{\infty}_{j=k+1} \nu \| \Delta x_j \|
\leq \sum^{\infty}_{j=k+1} 4\nu \normmm{(\rd F(\bar{x}))^\dag}\cdot\|F(x_j)\| \\
&=& 4\nu \normmm{(\rd F(\bar{x}))^\dag}
\sum^{\infty}_{j=0}\big( 1- t(1-\eta_k^{IN}) \big)^j \|F(x_{k+1})\| \\
&\leq&  4\nu \normmm{(\rd F(\bar{x}))^\dag}
\sum^{\infty}_{j=0}\big(1-t(1-\eta_{\max}) \big)^j \|F(x_{k+1})\| \\
&=& \frac{4\nu \normmm{(\rd F(\bar{x}))^\dag}}
{t(1-\eta_{\max})} \|F(x_{k+1})\|\\
&\le& c_2\frac{4\nu \normmm{(\rd F(\bar{x}))^\dag}}{t(1-\eta_{\max})}\dist^2(x_k, \bar{x}).
\end{eqnarray*}
This completes the proof.
\end{proof}

\begin{remark}
Let $\bar{x}\in\cm$  be an accumulation point of the sequence $\{x_k\}$ generated by  {\rm Algorithm \ref{dogleg1}}.
By {\rm Lemma \ref{lem:stepsize}} and the condition that $\delta_k \geq \delta_{\min}$,
if {\rm  Assumptions \ref{ASSUM} and \ref{ass:nons}} are satisfied,
then  $\Delta x_k^{IN}$ is a point contained in $\{\xi \in T_{x_k}\cm\; |\;\|\xi\|\le\delta_{\min}\}\subset \{\xi \in T_{x_k}\cm\; |\;\|\xi\|\le\delta_k\}$,
which also satisfies the Ared/Pred condition {\rm (\ref{Ared/Predcondition})}
for all $k$ sufficiently large.
Thus, if $\Delta x_k^{IN}$ is first tested for determining $\Delta x_k$ in {\rm Step 3}
of {\rm Algorithm \ref{dogleg1}}, then $x_{k+1}=R_{x_k}(\Delta x_k^{IN})$ for all $k$ sufficiently large.
Based on {\rm Theorem \ref{th:qc}}, the sequence $\{x_k\}$ converges to
$\bar{x}\in \cm$ quadratically.
\end{remark}

\section{Application in the SIEP } \label{sec4}

In this section, we apply the Riemannian inexact Newton dogleg method (Algorithm \ref{dogleg1}) to the SNIEP (\ref{RNEWEQ1}).
We also discuss the corresponding surjectivity  condition. Finally,  we study the associated preconditioning technique
for the SNIEP.

\subsection{Geometric properties}\label{sec41}
To apply Algorithm \ref{dogleg1} to solving the SNIEP (\ref{RNEWEQ1}), we need to derive the basic geometric properties of the product manifold $\SRn \times \co(n)$ and   the differential of $\Phi$ defined in (\ref{RNEWEQ1}).

We note that the tangent space of $\SRn \times \co(n)$ at a point $(S,Q)\in\SRn \times \co(n)$ is given by (see \cite[p. 42]{AMS08})
\[
\begin{array}{c}
T_{(S,Q)}\big(\SRn \times \co(n)\big)=  \{ (H, Q\Omega) \; | \; H^T = H, \Omega^T = -\Omega, \; H, \Omega \in \Rnn \}.
\end{array}
\]
Since $\SRn \times \co(n)$ is an embedded submanifold of $\SRn \times \Rnn$,
we can equip $\SRn \times \co(n)$ with
the following induced Riemannian metric:
\BE\label{METRIC2}
g_{(S,Q)}\big( (\xi_1, \eta_1), (\xi_2, \eta_2) \big):= \tr(\xi_1^T\xi_2) + \tr(\eta_1^T\eta_2),
\EE
for all $(S,Q) \in \SRn \times \co(n)$, $(\xi_1, \eta_1), (\xi_2, \eta_2) \in  T_{(S,Q)}(\SRn \times \co(n))$. Without causing any confusion, we still use $\langle \cdot, \cdot\rangle$ and $\|\cdot\|$ to denote the Riemannian metric on $\SRn \times \co(n)$ and its induced norm. Then the orthogonal projection of any $(\xi,\eta)\in\SRn \times \Rnn$ onto $T_{(S,Q)}(\SRn \times \co(n))$
is given by
\[
\Pi_{(S,Q)} (\xi, \eta) = \big( \xi, Q {\rm skew} (Q^T\eta) \big),
\]
where ${\rm skew}(A):= \frac{1}{2}(A-A^T)$.
A retraction on $\SRn \times \co(n)$ can be chosen as \cite[p.58]{AMS08}:
\BE\label{retr:sc}
R_{(S,Q)}(\xi_S, \eta_Q) = \big(S+\xi_S, {\rm qf} (Q + \eta_Q)\big), \quad {\rm for} \;
(\xi_S, \eta_Q) \in  T_{(S,Q)}\big(\SRn \times \co(n)\big),
\EE
where ${\rm qf}(A)$ denotes the $Q$ factor of an invertible matrix $A\in \mathbb{R}^{n\times n}$ as $A=\widehat{Q}\widehat{R}$,
where $\widehat{Q}$ belongs to $\co(n)$ and $\widehat{R}$ is an upper triangular matrix with strictly positive diagonal elements.

It is easy to verify that the differential $\rd \Phi(S,Q): T_{(S,Q)}(\SRn \times \co(n)) \to T_{\Phi(S,Q)}\SRn$ of $\Phi$ at a point $(S,Q)\in \SRn \times \co(n)$ is determined by
\BE\label{eq:NIEPDF1}
\rd \Phi(S,Q)[ (\Delta S,\Delta Q) ]= 2S\odot \Delta S + [ Q\Lambda Q^T ,\Delta QQ^T ],
\EE
for all $(\Delta S,\Delta Q) \in T_{(S,Q)}(\SRn \times \co(n)$. For any $Z\in\SRn$, we have $T_Z\SRn$ identifies $\SRn$ (i.e., $T_Z\SRn\simeq \SRn$).  Then,  $T_Z\SRn$ can be endowed  with the standard inner product on $\SRn$:
\BE\label{eq:SR-ip}
\langle \xi_Z,\eta_Z \rangle_F=\tr(\xi_Z^T\eta_Z), \quad\forall \xi_Z,\eta_Z\in T_Z\SRn
\EE
and its induced norm $\|\cdot\|_F$. Thus, with respect to the Riemannian metrics  (\ref{METRIC2}) and \eqref{eq:SR-ip}, the adjoint operator $(\rd \Phi(S,Q))^*: T_{\Phi(S,Q)}\SRn \to T_{(S,Q)}(\SRn \times \co(n)) $ of $\rd \Phi(S,Q)$ is determined by
\BE\label{eq:SNIEPDF2}
(\rd \Phi(S,Q))^*[\Delta Z] =\big(2S\odot \Delta Z ,  [ Q\Lambda Q^T , \Delta Z ]Q \big),
\quad \forall \Delta Z\in T_{\Phi(S,Q)}\SRn.
\EE

Based on the above analysis, we can use Algorithm \ref{dogleg1} to solving the SNIEP (\ref{RNEWEQ1}). On the convergence analysis of Algorithm \ref{dogleg1} for the SNIEP (\ref{RNEWEQ1}), we have the following remark.

\begin{remark}
The mapping $\Phi:\SRn \times \co(n)\to\SRn$ defined in {\rm (\ref{RNEWEQ1})}  satisfies the first conditions  of {\rm Assumption \ref{ASSUM}} since $\Phi$ is  a smooth mapping. The retraction $R$ defined by  \eqref{retr:sc} satisfies the second condition of {\rm Assumption \ref{ASSUM}} since $\co(n)$ is compact  {\rm \cite[p.149]{AMS08}} and $\SRn$ is a linear manifold. Thus, for the  {\rm SNIEP} {\rm (\ref{RNEWEQ1})},  {\rm Assumption \ref{ASSUM}} is satisfied.
\end{remark}

\subsection{Surjectivity condition}\label{sec32}
Let $(\overline{S},\overline{Q})\in \SRn \times \co(n)$ be an accumulation point of the sequence $\{(S_k,Q_k)\}$ generated by Algorithm \ref{dogleg1}  for solving the SNIEP (\ref{RNEWEQ1}).
To guarantee the the global and quadratic convergence of Algorithm \ref{dogleg1} for the SNIEP (\ref{RNEWEQ1}), we discuss the surjectivity condition of the differential $\rd \Phi(\cdot)$ at  $(\overline{S},\overline{Q})$.

Since $T_{\Phi(\overline{S},\overline{Q})}\SRn={\rm im}\big(\rd  \Phi(\overline{S},\overline{Q})\big) \oplus {\rm ker}\big( (\rd \Phi(\overline{S},\overline{Q}))^*  \big)$, the differential
$\rd  \Phi(\overline{S},\overline{Q})$ is surjective if and only if ${\rm ker}\big( (\rd \Phi(\overline{S},\overline{Q}))^*  \big) = \{{\bf 0}_{n\times n}\}$. This, together with \eqref{eq:SNIEPDF2}, implies that $\rd  \Phi(\overline{S},\overline{Q})$ is surjective if and only if the following linear matrix equation
\BE\label{eq:matrixeq1}
\left\{
\begin{array}{l}
\overline{S}\odot \Delta Z = \mathbf{0}_{n\times n},\\[2mm]
\overline{Q}\Lambda \overline{Q}^T\Delta Z - \Delta Z\overline{Q}\Lambda \overline{Q}^T = \mathbf{0}_{n\times n}
\end{array}
\right.
\EE
has  a unique solution $\Delta Z=\mathbf{0}_{n\times n}$. We note that there exists a unique linear transformation matrix $G\in\R^{n^2\times (n(n+1)/2)}$ such that
\BE\label{mat:G}
\ve(Z)=G\veh(Z),\quad\forall Z\in\SRn,
\EE
where $G$ is full column rank \cite{HS79}.
Then the matrix equation (\ref{eq:matrixeq1}) has  a unique solution $\Delta Z=\mathbf{0}_{n\times n}$ if and only if the following linear  equation
\BE\label{eq:matrixeq2}
\left\{
\begin{array}{l}
\mathrm{diag}(\mathrm{vec}(\overline{S})) G\Delta \bz = \mathbf{0}_{n^2},\\[2mm]
(\overline{Q} \otimes \overline{Q}) ( I_{n} \otimes \Lambda - \Lambda \otimes I_{n}) ( \overline{Q} \otimes \overline{Q})^T G\Delta \bz= \mathbf{0}_{n^2}.
\end{array}
\right.
\EE
has a unique solution $\Delta \bz= \mathbf{0}_{n(n+1)/2}\in\R^{n(n+1)/2}$, where $\mathbf{0}_{n}$ means the zero $n$-vector.

Therefore, we have the following result on the surjectivity of  $\rd  \Phi(\overline{S},\overline{Q})$.
\begin{theorem}\label{thm:sc}
Let $(\overline{S},\overline{Q})\in \SRn \times \co(n)$   be an accumulation point of the sequence $\{(S_k,Q_k)\}$ generated by
{\rm Algorithm  {\rm\ref{dogleg1}}} for solving the {\rm SNIEP \eqref{RNEWEQ1}}. Then the linear operator
$\rd  \Phi(\overline{S},\overline{Q})$ is surjective if and only if
\[
{\rm null}\left(\left[
\begin{array}{c}
\mathrm{diag}(\mathrm{vec}(\overline{S}))\\[2mm]
(\overline{Q}\otimes \overline{Q}) ( I_{n} \otimes \Lambda - \Lambda \otimes I_{n}) ( \overline{Q} \otimes \overline{Q})^T
\end{array}
\right]G\right)=\{\mathbf{0}_{n^2}\},
\]
where  $G\in\R^{n^2\times (n(n+1)/2)}$  is the linear transformation matrix defined by  \eqref{mat:G}.
\end{theorem}

On Theorem \ref{thm:sc}, we have the following remark.
\begin{remark}
Let
\[
J_{\overline{S}}:= \mathrm{diag}(\mathrm{vec}(\overline{S}))\quad\mbox{and}\quad
J_{\overline{Q}}:=(\overline{Q}\otimes \overline{Q}) ( I_{n} \otimes \Lambda - \Lambda \otimes I_{n}) ( \overline{Q} \otimes \overline{Q})^T
\]
and
\[
J_{(\overline{S},\overline{Q})}:=
\left[
\begin{array}{l}
J_{\overline{S}}\\[2mm]
J_{\overline{Q}}
\end{array}
\right].
\]
We note that
\[
\left\{
\begin{array}{l}
{\rm rank}(J_{\overline{S}}) =\mbox{number of nonzero elements of $\overline{S}$},\\[2mm]
{\rm rank}(J_{\overline{Q}}) = {\rm rank}( I_{n} \otimes \Lambda - \Lambda \otimes I_{n})
= n^2 - \sum\limits_{i=1}^n c_i,
\end{array}
\right.
\]
where $c_i$ is the multiplicity of $\lambda_i$ for $i=1,\ldots,n$. By {\rm Theorem {\rm\ref{thm:sc}}} and the fact that $G$ is full column rank, $\rd \Phi(\overline{S},\overline{Q})$ is surjective if and only if $J_{(\overline{S},\overline{Q})}$ is of full column rank. Specially, if the matrix $\overline{S}$ contains no zero elements, then the matrix $J_{\overline{S}}$ is full column rank and thus $J_{(\overline{S},\overline{Q})}G$ is full column rank.
\end{remark}

\subsection{Preconditioning technique} \label{sec33}
In this subsection, we consider the preconditioning technique   for solving the SNIEP (\ref{RNEWEQ1}) via Algorithm \ref{dogleg1}. When applying Algorithm \ref{dogleg1} to the SNIEP (\ref{RNEWEQ1}), we need to solve the following normal equation
\BE\label{NE:SNIEP}
\big(\rd \Phi(S_k,Q_k)\circ(\rd \Phi(S_k,Q_k))^* +\sigma_k \mathrm{id}_{T_{\Phi(S_k,Q_k)}\SRn}\big)
[ \Delta Z_k ] = -\Phi(S_k,Q_k)
\EE
for $\Delta Z_k\in T_{\Phi(S_k,Q_k)}\SRn$. To accelerate the convergence of the CG  method for solving \eqref{NE:SNIEP}, we solve the following left preconditioned linear equation
\BE\label{eq:PRESNPNEW1}
\begin{array}{l}
M_k^{-1}\circ \big(\rd \Phi(S_k,Q_k)\circ(\rd \Phi(S_k,Q_k))^* +\sigma_k \mathrm{id}_{T_{\Phi(S_k,Q_k)}\SRn}\big)
[ \Delta Z_k ] = -M_k^{-1}[\Phi(S_k,Q_k)],
\end{array}
\EE
where the preconditioner $M_k : T_{\Phi(S_k,Q_k)}\SRn \to T_{\Phi(S_k,Q_k)}\SRn$ is a self-adjoint and positive definite linear operator.

In the following, we construct an effective  preconditioner $M_k$. From \eqref{eq:NIEPDF1} and \eqref{eq:SNIEPDF2} we have,
for $\Delta Z_k\in T_{\Phi(S_k,Q_k)}\SRn$,
\begin{eqnarray}\label{eq:SNPNEW1}
H_k[ \Delta Z_k ] &:=& (\rd \Phi(S_k,Q_k)\circ (\rd \Phi(S_k,Q_k))^* + \sigma_k \mathrm{id}_{T_{\Phi(S_k,Q_k)} \SRn}) [ \Delta Z_k ] \nonumber\\
&=& 4S_k\odot S_k\odot \Delta Z_k + \big[ Q_k\Lambda Q_k^T, [ Q_k\Lambda Q_k^T, \Delta Z_k] \big] + \sigma_k \Delta Z_k.
\end{eqnarray}
Using \eqref{eq:SNPNEW1} we have
\[
\ve(H_k[\Delta Z]) = \widehat{H}_k \ve(\Delta Z),
\quad  \forall \Delta Z\in T_{\Phi(S_k,Q_k)}\SRn,
\]
where
\begin{eqnarray*}
\widehat{H}_k =  4\diag(\ve(S_k\odot S_k)) +  (Q_k \otimes Q_k)\big( (I_{n} \otimes \Lambda -\Lambda \otimes I_{n})^2 + \sigma_kI_{n^2}\big)(Q_k \otimes Q_k)^T.
\end{eqnarray*}

Then we can construct a preconditioner $M_k$ such that
\BE\label{prec:mk}
M_k[\Delta Z]: = (s_k + \sigma_k)\Delta Z + \big[ Q_k\Lambda Q_k^T, [ Q_k\Lambda Q_k^T, \Delta Z] \big],
\quad \forall \Delta Z\in T_{\Phi(S_k,Q_k)}\SRn,
\EE
where $s_k: = \max\{ 4(S_k\odot S_k)_{ij}, i,j=1,\ldots,n\}$.
Using \eqref{prec:mk} we obtain
\[
\ve(M_k[\Delta Z]) = \widehat{M}_k \ve(\Delta Z),
\quad  \forall \Delta Z\in T_{\Phi(S_k,Q_k)}\SRn,
\]
where
\begin{eqnarray*}
\widehat{M}_k =(Q_k \otimes Q_k)\big( (I_{n} \otimes \Lambda -\Lambda \otimes I_{n})^2 + (s_k + \sigma_k)I_{n^2}\big)(Q_k \otimes Q_k)^T.
\end{eqnarray*}
To compute $M_k^{-1}[\Delta Z]$ for all $\Delta Z\in T_{\Phi(S_k,Q_k)}\SRn$,
we note that the matrix $\widehat{M}_k$ is real symmetric and positive definite
and its inverse is given by
\[
\widehat{M}_k^{-1} = (Q_k \otimes Q_k)\big( (I_{n} \otimes \Lambda - \Lambda \otimes I_{n})^2
+(\sigma_k + \overline{s}_k )I_{n^2} \big)^{-1}(Q_k \otimes Q_k)^T,
\]
which can be computed readily. Thus,
\[
M_k^{-1}[\Delta Z]= \ve^{-1}\Big(\widehat{M}_k^{-1} \ve(\Delta Z) \Big),
\quad  \forall \Delta Z\in T_{\Phi(S_k,Q_k)}\SRn.
\]
is available readily since the matrix-vector product $\widehat{M}_k^{-1} \ve(\Delta Z)$ can be computed efficiently.

\section{Numerical experiments} \label{sec5}

In this section, we report numerical performance of Algorithm \ref{dogleg1} for solving the SNIEP (\ref{RNEWEQ1}).
To show the efficiency of the proposed preconditioner, we compare Algorithm {\rm \ref{dogleg1}} with
the Riemannian inexact Newton method (RIN) \cite{ZBJ18}. All numerical tests are obtained
using {\tt MATLAB} R2020a on a linux server (20-core, Intel(R) Xeon (R) Gold 6230 @ 2.10 GHz, 32 GB RAM).

To determine $\Delta x_k \in \Gamma^{DL}_k$ such that $\min\{\delta_{\min},\|\Delta x_k^{IN}\|\} \leq \|\Delta x_k\| \leq \delta_k$
in Steps 2 and 3 of Algorithm \ref{dogleg1}, the  following traditional strategy is used.

\begin{procedure} {\rm (Determination of $\Delta x_k$)}
\begin{description}
\item if $\|\Delta x_k^{IN}\|\leq\delta_k$  then set $\Delta x_k:=\Delta x_k^{IN}$.
\item else if $\|\widehat{\Delta x}_k^{CP}\|\geq\delta_k$ then set $\Delta x_k:=\frac{\delta_k}{\|\widehat{\Delta x}_k^{CP}\|}\widehat{\Delta x}_k^{CP}$,
\item else set $\Delta x_k:=(1-\gamma)\widehat{\Delta x}_k^{CP}+\gamma\Delta x_k^{IN}$  for $\gamma \in (0,1)$ such that  $\|\Delta x_k\| = \delta_k$.
\item endif
\end{description}
\end{procedure}

For the determination of $\delta_{k+1}$ in Step 4 of Algorithm \ref{dogleg1},
we make use of the following special strategy \cite[p.2126]{PS08}.
\begin{procedure} {\rm (Determination of $\delta_{k+1}$)}
\begin{description}
\item if $\frac{{\rm Ared}_{k}(\Delta x_k)}{{\rm Pred}_{k}(\Delta x_k)}< \rho_s$ then
\begin{description}
\item   if $\|\Delta x_k^{IN}\|<\delta_k$ then set $\delta_{k+1}:=\max\{ \|\Delta x_k^{IN}\|, \delta_{\min}\}$,
\item else then set $\delta_{k+1}:=\max\{ \beta_s \delta_k, \delta_{\min}\}$.
\end{description}
\item else  $\frac{{\rm Ared}_{k}(\Delta x_k)}{{\rm Pred}_{k}(\Delta x_k)}\ge \rho_s$ then
\begin{description}
\item if $\frac{{\rm Ared}_{k}(\Delta x_k)}{{\rm Pred}_{k}(\Delta x_k)} >  \rho_e$ and $\|\Delta x_k\| = \|\delta_k\|$ then set $\delta_{k+1}:=\min\{ \beta_e \delta_k, \delta_{\max}\}$.
\end{description}
\end{description}
\end{procedure}

In our numerical tests, we set $t = 10^{-4}$,  $\sigma_{\max}=10^{-6}$, $\theta_{\min}=0.1$, $\theta_{\max}=0.9$,
$\delta_{\min} = 10^{-8}$, $\delta_{\max} = 10^{10}$, $\rho_s= 0.1$, $\rho_e = 0.75$,
$\beta_s= 0.25$ and $\beta_e = 4.0$. In addition, we set $\theta_k = 0.25$, and $\bar{\eta}_k = \frac{1}{k+10}$ for all $k\ge 0$.
The initial value of $\delta_0$ is set as follows: If $\|\Delta x_0^{IN}\| < \delta_{\min}$, set $\delta_0 = 2\delta_{\min}$;
else $\delta_0 = \|\Delta x_0^{IN}\|$.
The parameters for the RIN are set as in \cite{ZBJ18}.
The stopping criteria for Algorithm {\rm \ref{dogleg1}} and the RIN for solving the SNIEP (\ref{RNEWEQ1}) are set to be
\[
\|\Phi(S_k,Q_k)\|_F \leq 5.0\times 10^{-10}.
\]

For Algorithm {\rm \ref{dogleg1}} and the RIN, we solve \eqref{NE:SNIEP} via the CG method
and preconditioned CG (PCG) method with the preconditioner $M_k$ defined in \eqref{prec:mk}.
The largest number of outer iterations is set to be 100 and the largest number of
inner CG iterations is set to be $n^2$.

In our numerical tests,  `{\tt CT.}', {\tt IT.}', `{\tt NF.}',  `{\tt NCG.}', and `{\tt Res.}' mean
the  total computing time in seconds, the number of outer iterations, the number of function evaluations,
the  number of inner CG iterations,  the residual $\|\Phi(S_k,Q_k)\|_F$
at the final iterates of the corresponding algorithms, accordingly.
In addition,  `{\tt Res0.}' denotes the  residual $\|\Phi(S_0,Q_0)\|_F$
at the initial iterates of the corresponding algorithms.

We first consider the following small example.
\begin{example}\label{ex:3}
We consider the SNIEP with the spectrum $\{5,0,-2,-2\}$ {\rm \cite{CD91,S83}}. We report our numerical results for
different starting points {\rm(}which are generated by  the {\tt MATLAB} built-in functions {\tt rand} and {\tt orth}{\rm)}{\rm :}
{\rm(a)} $S_0=(B + B')/2$ with $B = {\tt rand}\,(n,n)$ and $Q_0={\tt orth}({\tt rand}(4,4))$, {\rm (b)} $S_0=(B + B')/2$ with $B = 5*{\tt rand}\,(4,4)$ and $Q_0={\tt orth}(5*{\tt rand}(4,4))$, and {\rm (c)}  $S_0=(B + B')/2$ with $B = 10*{\tt rand}\,(4,4)$ and $Q_0={\tt orth}(10*{\tt rand}(4,4))$.
\end{example}

We apply the RIN and  Algorithm  \ref{dogleg1} to Example \ref{ex:3}. The computed solution to the SNIEP via Algorithm  \ref{dogleg1} with PCG is as follows: For Case (a),
\[
\overline{C}=\left[
\begin{array}{cccc}
    0.6347  &  1.8878  &  2.2597  &  1.6700 \\
    1.8878  &  0.2945  &  1.3510  &  0.2270 \\
    2.2597  &  1.3510  &  0.0144  &  1.7082 \\
    1.6700  &  0.2270  &  1.7082  &  0.0565
\end{array}
\right];
\]
for Case (b),
\[
\overline{C}=\left[
\begin{array}{cccc}
    0.4120  &  0.9163  &  1.3446  &  2.2396 \\
    0.9163  &  0.2899  &  2.1531  &  1.2448 \\
    1.3446  &  2.1531  &  0.1386  &  1.5818 \\
    2.2396  &  1.2448  &  1.5818  &  0.1595
\end{array}
\right];
\]
for Case (c),
\[
\overline{C}=\left[
\begin{array}{cccc}
    0.0951  &  1.2360  &  2.0772  &  1.9702 \\
    1.2360  &  0.6101  &  0.6151  &  1.7514 \\
    2.0772  &  0.6151  &  0.2576  &  1.7623 \\
    1.9702  &  1.7514  &  1.7623  &  0.0373
\end{array}
\right].
\]

The numerical results for   Example \ref{ex:3} are given in Table \ref{table3}. We see from  Table \ref{table3} that both
the RIN and  Algorithm  \ref{dogleg1}  can find a solution to the SNIEP effectively.

\begin{table}[!h]\renewcommand{\arraystretch}{1.1} \addtolength{\tabcolsep}{1pt}
  \caption{Numerical results of Example \ref{ex:3}.}\label{table3}
  \begin{center} {\scriptsize
   \begin{tabular}[c]{|c|c|c|c|c|c|c|c|}
     \hline
     \multicolumn{8}{|c|}{ Example \ref{ex:3}} \\ \hline
Alg.     & Case      & {\tt CT.}     & {\tt IT.}    & {\tt NF.} & {\tt NCG.} &  {\tt Res0.} &  {\tt Res.} \\ \hline
RIN      & (a)       &   0.0013 s    &   6  &   8   &   7  &   4.8290  &  $9.87\times 10^{-13}$ \\
with     & (b)       &   0.0031 s    &   6  &   7   &   7  &   26.456  &  $4.04\times 10^{-12}$ \\
CG       & (c)       &   0.0031 s    &   8  &   9   &   6  &   175.37  &  $2.04\times 10^{-13}$ \\
             \cline{2-7}\hline
RIN      & (a)       &   0.0013 s    &   6  &   8   &   5  &  4.8290   &  $7.40\times 10^{-12}$\\
with     & (b)       &   0.0027 s    &   7  &   8   &   6  &  26.456   &  $1.80\times 10^{-15}$ \\
PCG      & (c)       &   0.0030 s    &   9  &  10   &   5  &  175.37   &  $4.65\times 10^{-15}$ \\
             \cline{2-7}\hline
Alg. 2.1 & (a)       &   0.0013 s    &   7  &   9   &   8  &   4.8290  &   $2.65\times 10^{-15}$\\
with     & (b)       &   0.0115 s    &   6  &   7   &   7  &   26.456  &   $5.94\times 10^{-12}$ \\
CG       & (c)       &   0.0048 s    &   8  &   9   &   6  &   175.37  &   $6.16\times 10^{-13}$ \\
             \cline{2-7}\hline
Alg. 2.1 & (a)       &   0.0012 s    &   6  &   8   &   5  &  4.8290   &   $6.15\times 10^{-13}$\\
with     & (b)       &   0.0051 s    &   6  &   7   &   5  &  26.456   &   $7.54\times 10^{-11}$ \\
PCG      & (c)       &   0.0044 s    &   8  &   9   &   5  &  175.37   &   $2.84\times 10^{-13}$ \\
             \cline{2-7}\hline
\end{tabular} }
  \end{center}
\end{table}

Next, we consider the SNIEP with arbitrary prescribed eigenvalues.
\begin{example}\label{ex:1}
We consider the SNIEP with arbitrary prescribed eigenvalues. Let $\widehat{C}$ be an $n\times n$ random symmetric nonnegative matrix generated by
the {\tt MATLAB} built-in functions {\tt randn} and {\tt abs}{\rm :}
\[
\widehat{C} = (\widetilde{C} + \widetilde{C}^T)/2\quad\mbox{with}\quad \widetilde{C} = {\tt abs}({\tt randn} (n,n)).
\]
We use the eigenvalues of $\widehat{C}$ as the prescribed spectrum.
The starting point $(S_0, Q_0)$ is generated as follows{\rm :}
\[
B = {\tt rand}\,(n,n), \quad C_0 = (B + B')/2, \quad
S_0 = {\tt sqrt}\,(C_0), \quad [ Q_0, \widetilde{\Lambda}] = \mbox{\tt eig}\,(C_0).
\]
\end{example}
\begin{example}\label{ex:2}
We consider the SNIEP with multiple zero eigenvalues. Let $\widehat{C}=XX^T$, where $X\in\R^{n\times p}$ is a random nonnegative matrix generated by the {\tt MATLAB} built-in function {\tt rand}.  We use the eigenvalues of $\widehat{C}$ as the prescribed spectrum. We choose the starting point $(S_0, Q_0)$ as follows{\rm :}
\[
B = {\tt rand}\,(n,p), \quad C_0 = B*B', \quad S_0 = {\tt sqrt}\,(C_0), \quad [ Q_0, \widetilde{\Lambda}] = \mbox{\tt eig}\,(C_0).
\]
\end{example}

Tables \ref{table1}--\ref{table2}  list numerical results for Examples \ref{ex:1} and \ref{ex:2}, respectively.
We observe from Tables \ref{table1}--\ref{table2} that both Algorithm {\rm \ref{dogleg1}} and the RIN are globally convergent.
In particular, the constructed preconditioner $M_k$ can improve the performances of these algorithms efficiently in terms of the computing time and the number of inner CG iterations.

\begin{table}[!h]\renewcommand{\arraystretch}{1.1} \addtolength{\tabcolsep}{1pt}
  \caption{Numerical results of Example \ref{ex:1}.}\label{table1}
  \begin{center} {\small
   \begin{tabular}[c]{|c|c|c|c|c|c|c|c|}
     \hline
Alg.     & $n$         & {\tt CT.}     & {\tt IT.}    & {\tt NF.} & {\tt NCG.} &  {\tt Res0.} &  {\tt Res.} \\ \hline
RIN      & $100$       &   0.2499 s    &   7  &   8   &  112 &   40.306  &  $2.08\times 10^{-13}$ \\
with     & $200$       &   0.7878 s    &   7  &   8   &  148 &   78.387  &  $3.90\times 10^{-13}$ \\
CG       & $500$       &   3.8080 s    &   7  &   8   &  192 &   194.82  &  $5.24\times 10^{-12}$ \\
         & $1000$      &   36.364 s    &   8  &   9   &  332 &   388.30  &  $2.92\times 10^{-12}$ \\
         & $2000$      &  04 m 28 s    &   8  &   9   &  402 &   788.42  &  $7.17\times 10^{-12}$ \\
         & $5000$      &01 h 18 m 22 s &   9  &  10   &  594 &   1945.0  &  $2.15\times 10^{-11}$ \\
             \cline{2-7}\hline
RIN      & $100$       &   0.0191 s    &   6  &   7   &   5  &  40.306   &  $3.99\times 10^{-12}$\\
with     & $200$       &   0.0555 s    &   6  &   7   &   6  &  78.387   &  $4.30\times 10^{-13}$ \\
PCG      & $500$       &   0.3356 s    &   6  &   7   &   5  &  194.82   &  $1.67\times 10^{-12}$ \\
         & $1000$      &   1.6659 s    &   7  &   8   &   5  &  388.30   &  $2.89\times 10^{-12}$ \\
         & $2000$      &   8.5646 s    &   7  &   8   &   5  &  788.42   &  $6.82\times 10^{-12}$ \\
         & $5000$      &  01 m 19 s    &   7  &   8   &   4  &  1945.0   &  $2.17\times 10^{-11}$ \\
             \cline{2-7}\hline
Alg. 2.1 & $100$       &   0.1588 s    &   6  &   7   &   84 &   40.306  &   $5.60\times 10^{-11}$\\
with     & $200$       &   0.8950 s    &   7  &   8   &  164 &   78.387  &   $3.97\times 10^{-13}$ \\
CG       & $500$       &   4.3912 s    &   7  &   8   &  219 &   194.82  &   $1.19\times 10^{-12}$ \\
         & $1000$      &   27.093 s    &   7  &   8   &  276 &   388.30  &   $6.83\times 10^{-11}$ \\
         & $2000$      &  05 m 02 s    &   8  &   9   &  447 &   788.42  &   $7.01\times 10^{-12}$ \\
         & $5000$      &01 h 37 m 45 s &   9  &  10   &  725 &   1945.0  &   $2.17\times 10^{-11}$ \\
             \cline{2-7}\hline
Alg. 2.1 & $100$       &   0.0278 s    &   6  &   7   &   5  &  40.306   &   $6.26\times 10^{-13}$\\
with     & $200$       &   0.0572 s    &   6  &   7   &   6  &  78.387   &   $3.48\times 10^{-13}$ \\
PCG      & $500$       &   0.3084 s    &   6  &   7   &   5  &  194.82   &   $1.58\times 10^{-12}$ \\
         & $1000$      &   1.8318 s    &   7  &   8   &   5  &  388.30   &   $2.88\times 10^{-12}$ \\
         & $2000$      &   9.9037 s    &   7  &   8   &   5  &  788.42   &   $6.82\times 10^{-12}$ \\
         & $5000$      &  01 m 27 s    &   7  &   8   &   4  &  1945.0   &   $5.62\times 10^{-11}$ \\
             \cline{2-7}\hline
\end{tabular} }
  \end{center}
\end{table}
\begin{table}[!h]\renewcommand{\arraystretch}{1.1} \addtolength{\tabcolsep}{1pt}
  \caption{Numerical results of Example \ref{ex:2}.}\label{table2}
  \begin{center} {\small
   \begin{tabular}[c]{|c|c|c|c|c|c|c|c|c|}
     \hline
Alg.     & $n$   & $p$     & {\tt CT.}     & {\tt IT.}    & {\tt NF.} & {\tt NCG.} &  {\tt Res0.} &  {\tt Res.} \\ \hline
RIN     & $100$   & $25$  &   0.0697 s    &  6  &  7   &  33 &   49.526  &  $1.31\times 10^{-12}$ \\
with    & $200$   & $50$  &   0.2519 s   &  6  &  7   &  50 &   43.667  &  $6.59\times 10^{-12}$ \\
CG     & $500$   &$125$ &   1.7630 s    &  7  &  8  &  75 &    185.69  &  $4.63\times 10^{-11}$ \\
         & $1000$  &$250$ &   10.981 s    &   6  &  7   &  116 &  25.570  &  $2.12\times 10^{-10}$ \\
         & $2000$  &$500$ &  01 m 05 s    &   6  &   7   &  125 &   111.58  &  $1.07\times 10^{-9}$ \\
         & $5000$  &$1250$ & 18 m 51 s &   6  &  7   &  210 &   77.947  &  $8.23\times 10^{-9}$ \\
             \cline{2-8}\hline
RIN      & $100$  & $25$  &   0.0195 s    &   5  &  6   &   5  &  49.526  &  $1.24\times 10^{-12}$\\
with     & $200$  & $50$ &   0.0426 s    &   5  &  6   &   5  &  43.667   &  $6.52\times 10^{-12}$ \\
PCG     & $500$ &$125$ &   0.2884 s    &   6  &  7  &    4  &  185.69  &  $3.86\times 10^{-11}$ \\
         & $1000$    &$250$ &   0.9763 s    &  5  &   6   &   4  &  25.570   &  $2.24\times 10^{-10}$ \\
         & $2000$   &$500$ &   4.5024 s    &   5  &   6   &   3  &  111.58   &  $1.04\times 10^{-9}$ \\
         & $5000$   &$1250$ &  53.165 s    &   5  &   6   &   3  &  77.947   &  $8.42\times 10^{-9}$ \\
             \cline{2-8}\hline
Alg. 2.1 & $100$  & $25$  &   0.0753 s    &  6  &  7   &   33 &   49.526  &   $1.24\times 10^{-12}$\\
with     & $200$  & $50$  &   0.2867 s   &  6  &  7   &   55 &   43.667   &   $6.64\times 10^{-12}$ \\
CG      & $500$ &$125$ &    1.9218 s    &  7  &  8   &     81 &   185.69  &   $4.38\times 10^{-11}$ \\
         & $1000$   &$250$ &   11.800 s    &   6  &   7   &  123 &   25.570  &   $2.12\times 10^{-10}$ \\
         & $2000$  &$500$ &  01 m 12 s    &   6  &   7   &  132 &   111.58  &   $1.00\times 10^{-9}$ \\
         & $5000$   &$1250$ & 19 m 51 s &   6  &  7   &  218 &   77.947  &   $8.05\times 10^{-9}$ \\
             \cline{2-8}\hline
Alg. 2.1 & $100$   & $25$  &   0.0208 s    &   5  &  6   &   5  &  49.526  &   $1.18\times 10^{-12}$\\
with     & $200$   & $50$  &   0.0464 s    &  5  &  6  &   5  &  43.667   &   $5.80\times 10^{-12}$ \\
PCG    & $500$   &$125$ &   0.3198 s    &   6  &   7   &   4  &  185.69 &   $4.49\times 10^{-11}$ \\
         & $1000$  &$250$ &   1.1114 s    &   5  &   6   &   4  &  25.570   &   $2.13\times 10^{-10}$ \\
         & $2000$  &$500$ &   4.9578 s    &   5  &   6   &   3  &  111.58   &   $1.10\times 10^{-9}$ \\
         & $5000$  &$1250$ &  01 m 01 s    &   5  &   6   &   3  &  77.947   &   $8.52\times 10^{-9}$ \\
             \cline{2-8}\hline
\end{tabular} }
  \end{center}
\end{table}

To illustrate the quadratic convergence of Algorithm {\rm \ref{dogleg1}}, we give the convergence trajectory
for two tests of Example \ref{ex:1} with $n=200$ and $n=1000$.
Figure \ref{f1} depicts the logarithm of the residual versus the number of iterations of
Algorithm \ref{dogleg1} and the RIN. We observe from Figure \ref{f1} that both Algorithm \ref{dogleg1} and the RIN converge quadratically, which confirms our theoretical results.

To further illustrate the efficiency of the preconditioner, we give the condition number and the spectrum of the matrices
$\widehat{H}_k$ and $\widehat{M}_k^{-1}\widehat{H}_k$ at the final iterates generated by
Algorithm {\rm \ref{dogleg1}} and the RIN for one test of Example \ref{ex:1} with $n=100$.
For the RIN, the condition numbers of $\widehat{H}_k$ and $\widehat{M}_k^{-1}\widehat{H}_k$
are  $5.01\times 10^3$ and $4.0665$, respectively, while, for Algorithm {\rm \ref{dogleg1}}, the condition numbers of $\widehat{H}_k$ and $\widehat{M}_k^{-1}\widehat{H}_k$
are  $5.01\times 10^3$ and $4.0658$, respectively.
Thus the preconditioner $\widehat{M}_k$ can reduce the condition number of $\widehat{H}_k$ efficiently.
From Figure \ref{f2}, we observe that the eigenvalues of $\widehat{H}_k$ are scattered in the interval $(0,8000)$,
while the eigenvalues of $\widehat{M}_k^{-1}\widehat{H}_k$ are clustered around 1. This shows the effectiveness of the constructed preconditioner.

\begin{figure}[!h]
\begin{center}
\begin{tabular}{cc}
\epsfig{figure=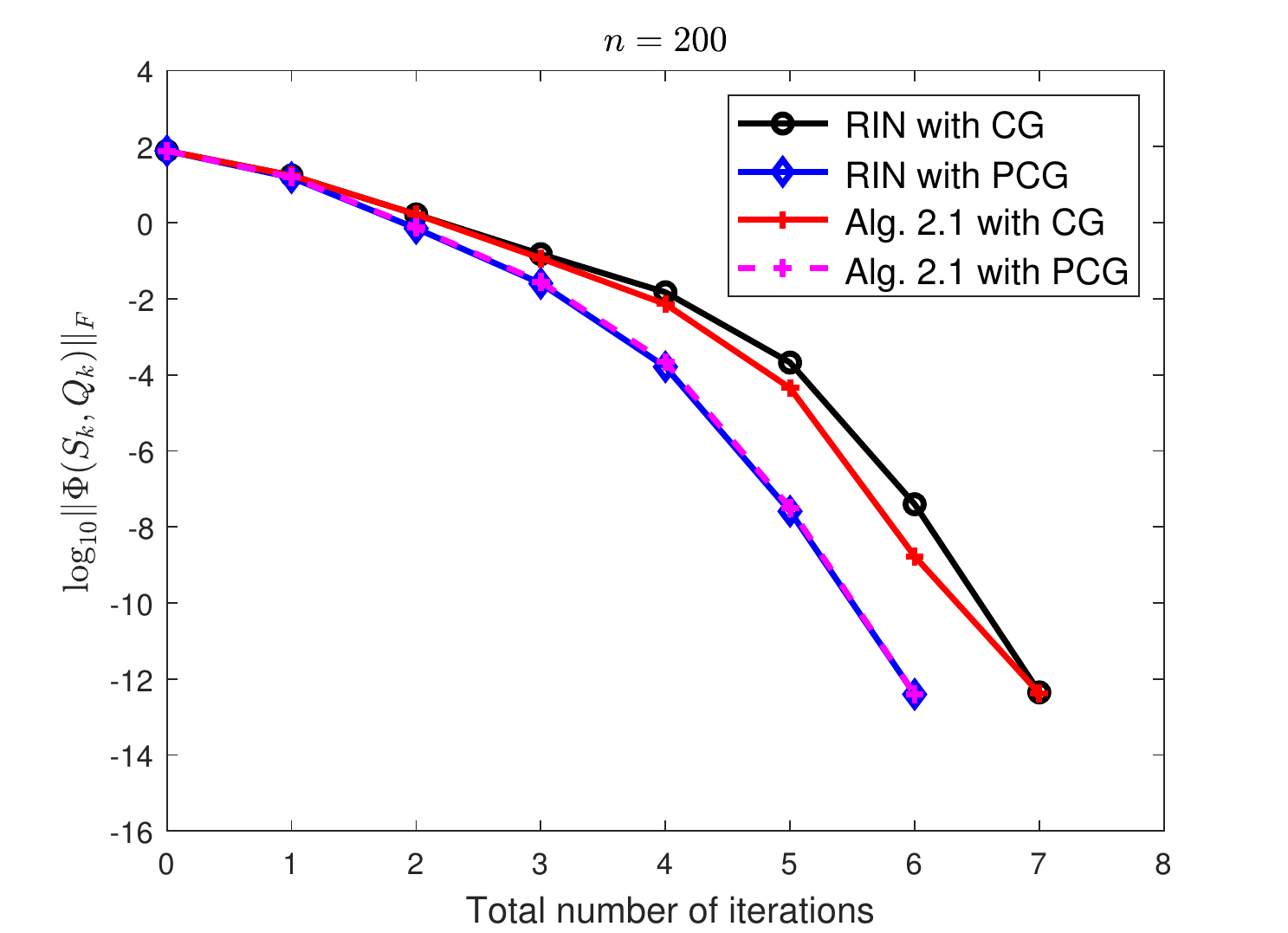,width=7.0cm} & \epsfig{figure=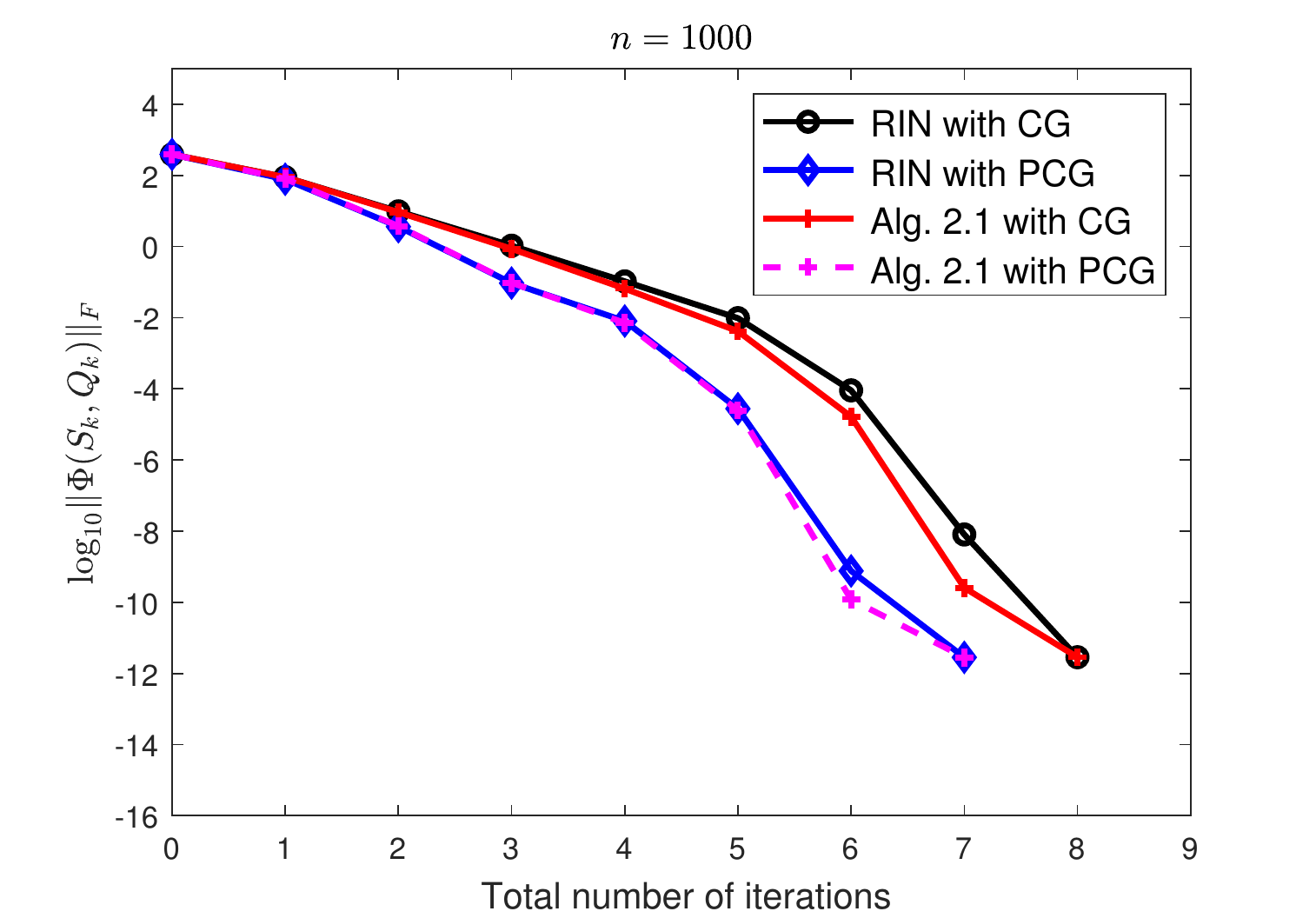,width=7.0cm} \\
\end{tabular}
\end{center}
 \caption{Convergence history of two tests for Example \ref{ex:1}.}
 \label{f1}
\end{figure}

\begin{figure}[!ht]
\begin{center}
\begin{tabular}{cc}
\epsfig{figure=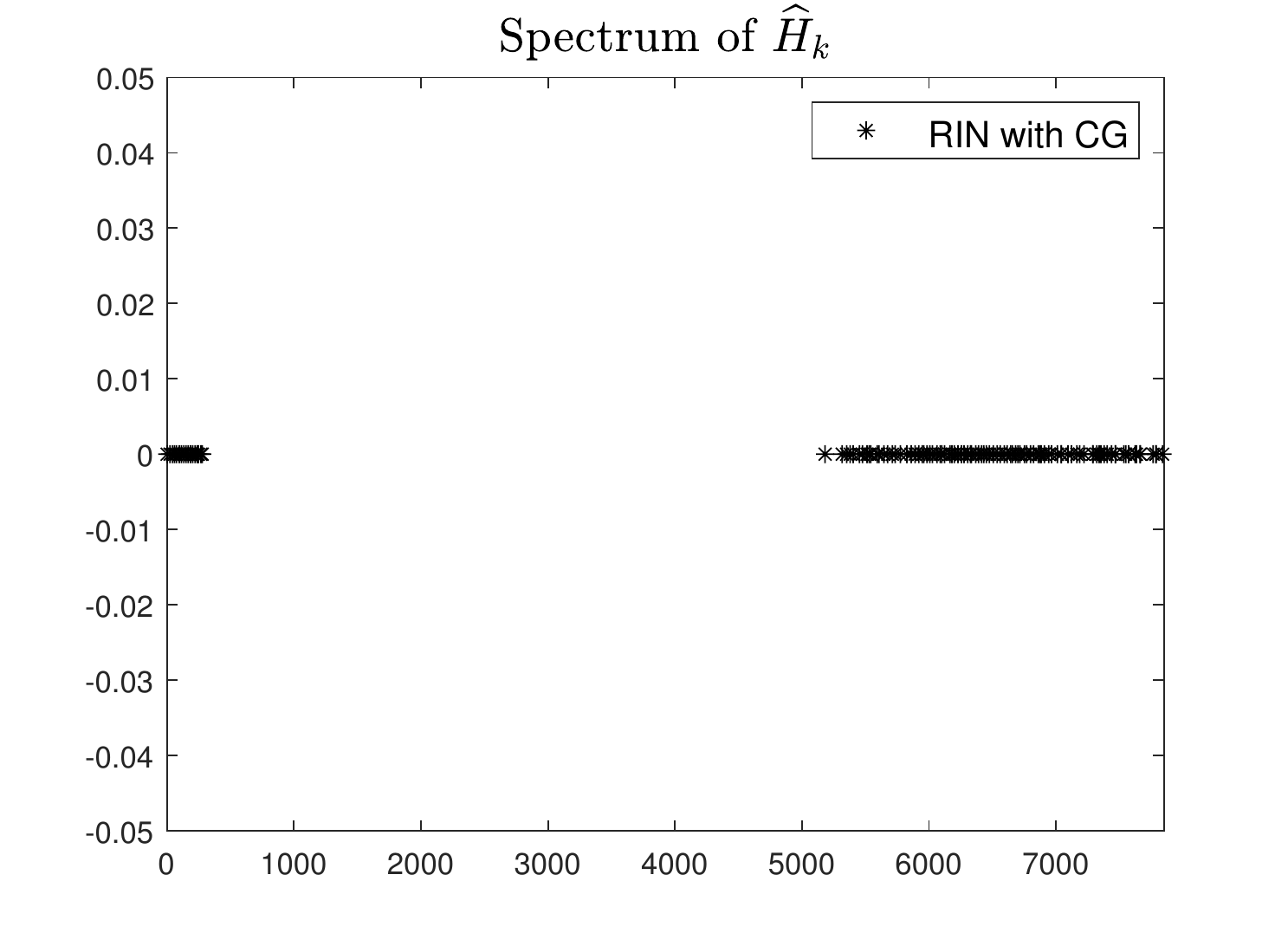,width=7.0cm} & \epsfig{figure=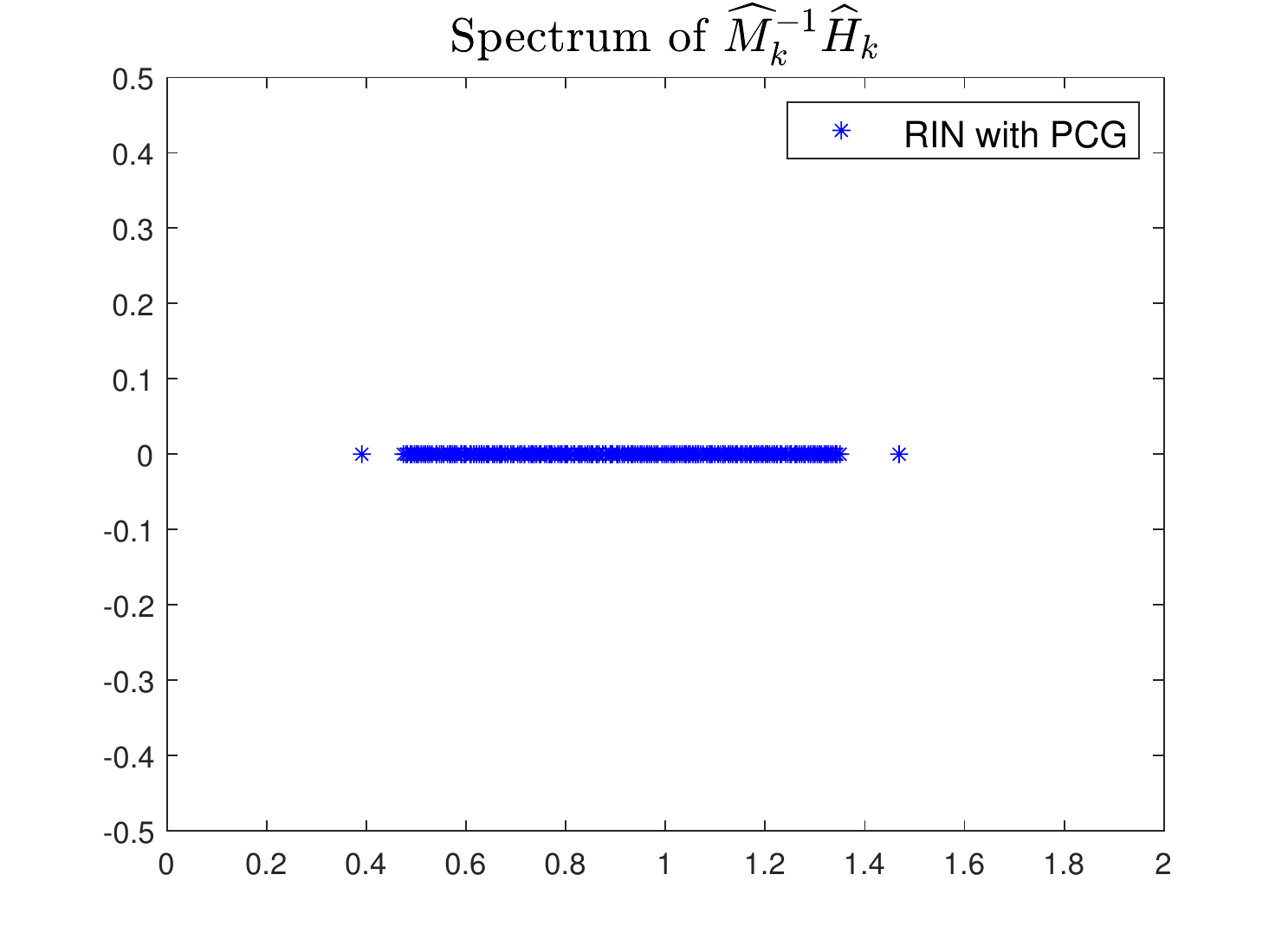,width=7.0cm} \\
\epsfig{figure=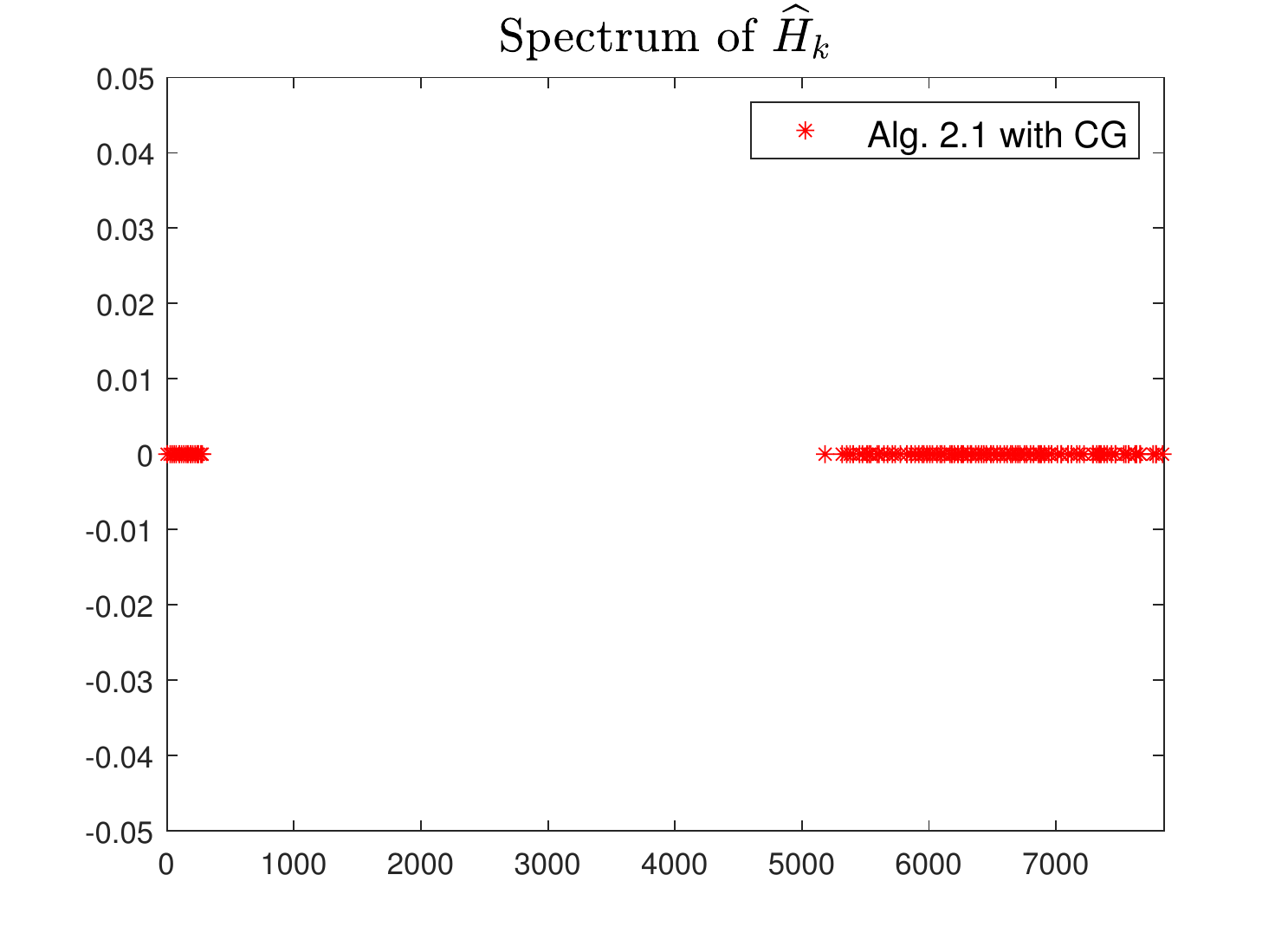,width=7.0cm} & \epsfig{figure=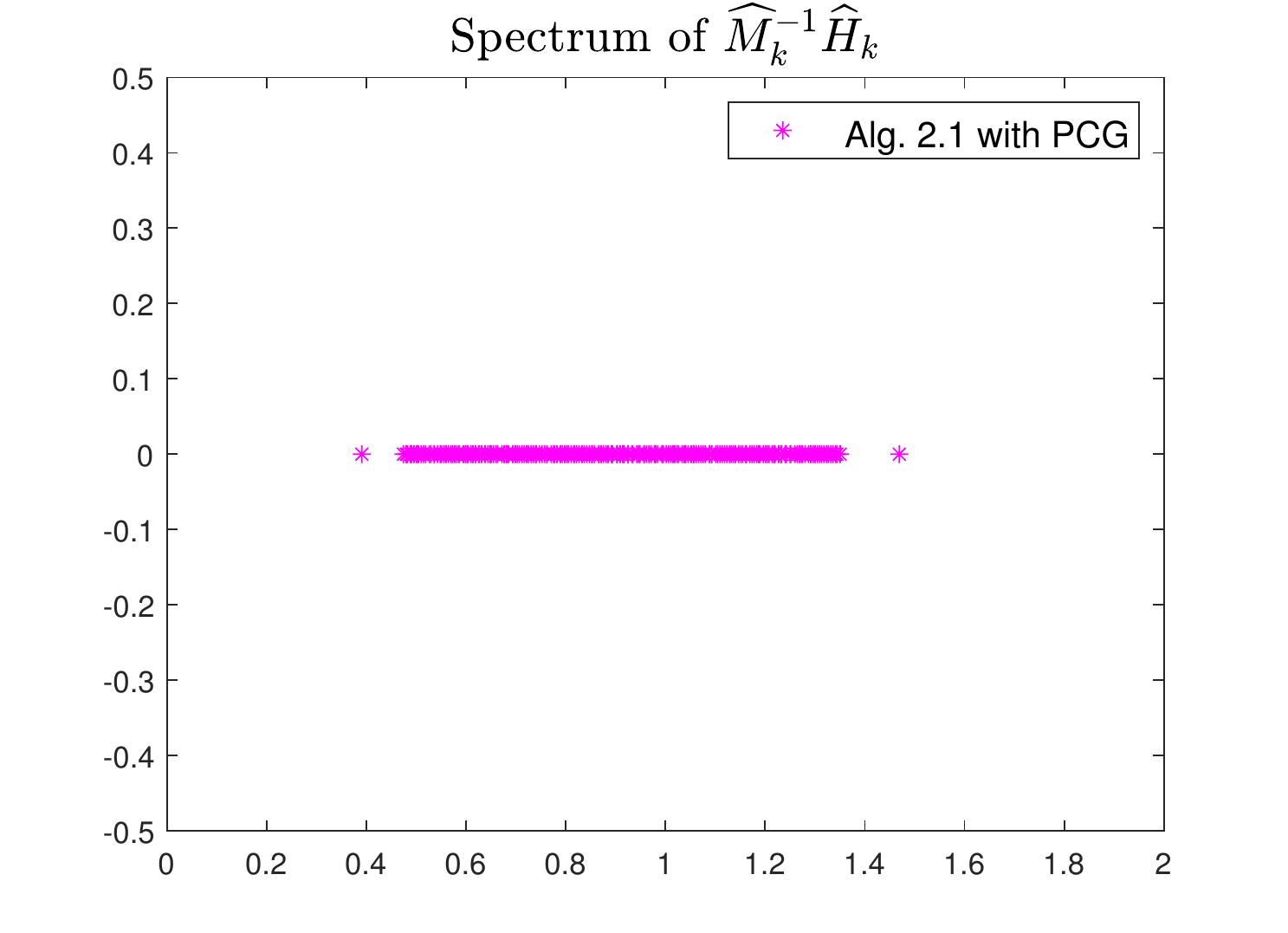,width=7.0cm}
\end{tabular}
\end{center}
 \caption{Spectrum of $\widehat{H}_k$ and $\widehat{M}_k^{-1}\widehat{H}_k$ at final iterates of one test for $n=100$.}
 \label{f2}
\end{figure}

\section{Concluding remarks} \label{sec6}

In this paper, we consider the problem of reconstructing a symmetric nonnegative matrix
from prescribed realizable spectrum. The inverse problem is reformulated as an underdetermined nonlinear matrix equation over a Riemannian product manifold. To solve the inverse problem, we develop a Riemannian underdetermined Newton dogleg method
for finding a solution to  a general underdetermined nonlinear equation defined between Riemannian manifold and Euclidean space.
Under some mild assumptions, we show the proposed method converges globally  and quadratically.
Then we apply he proposed method to inverse problem by constructing an efficient preconditioner.
Numerical results show the efficiency of the proposed method.
In the future research, we will discuss how to construct an effective preconditioned
numerical method for solving the inverse eigenvalue problem for nonsymmetric nonnegative matrices.




\end{document}